\documentclass[11pt]{article}
\usepackage{fullpage}

\usepackage[utf8]{inputenc}
\usepackage{amsthm,amsfonts,amstext,amssymb,mathrsfs,amsmath,latexsym,mathtools} 
\usepackage{graphicx}
\usepackage{blindtext}
\usepackage{csquotes}
\usepackage[bordercolor=gray!20,backgroundcolor=blue!20,linecolor=blue,textsize=footnotesize,textwidth=0.8in,textcolor=blue]{todonotes}
\usepackage{enumerate}
\usepackage{bm}
\usepackage{hyperref}
\usepackage[margin=0.5cm]{caption} 
\usepackage{seqsplit}
\usepackage{dsfont}
\usepackage{graphicx}
\usepackage{relsize}

\usepackage[nocompress]{cite}

\usepackage{tikz}
\usetikzlibrary{shapes.geometric, arrows}
\tikzstyle{hipotese} = [rectangle, rounded corners, minimum width=3cm, minimum height=1cm,text centered, draw=black, fill=red!30]
\tikzstyle{tese} = [diamond, minimum width=3cm, minimum height=1cm, text centered, draw=black, fill=green!30]
\tikzstyle{arrow} = [thick,->,>=stealth]

\usepackage{xcolor}
\usepackage{mathtools}
\usepackage{hyperref}
\usepackage{cleveref}
\usepackage{enumerate}
\usepackage{dsfont}
\usepackage{pdfpages}
\usepackage{ mathrsfs }

\newtheorem{theorem}{Theorem}
\numberwithin{theorem}{section}

\newtheorem{proposition}[theorem]{Proposition}

\newtheorem{corollary}[theorem]{Corollary}

\newtheorem{lemma}[theorem]{Lemma}

\theoremstyle{definition}
\newtheorem{definition}[theorem]{Definition}
\newtheorem{example}[theorem]{Example}

\theoremstyle{remark}
\newtheorem*{remark}{Remark}

\newcommand{\C}{\mathbb{C}}
\newcommand{\R}{\mathbb{R}}
\newcommand{\N}{\mathbb{N}}

\newcommand{\Mcal}{\mathcal{M}}
\newcommand{\Pcal}{\mathcal{P}}
\newcommand{\Wcal}{\mathcal{W}}
\newcommand{\Tcal}{\mathcal{T}}

\newcommand{\Ccal}{\mathcal{C}}

\newcommand{\Fcal}{\mathcal{F}}
\newcommand{\Ord}{\mathcal{O}}

\newcommand{\Zcal}{\mathcal{Z}}
\newcommand{\dd}{\mathrm{d}}

\newcommand{\ee}{\mathrm{e}}
\newcommand{\ii}{\mathrm{i}}

\newcommand{\Leb}{\mathrm{m}_2}
\newcommand{\wh}{\widehat}
\newcommand\restr[2]{{#1}\raisebox{-.3ex}{$|$}_{#2}}
\newcommand*{\defeq}{\mathrel{\vcenter{\baselineskip0.5ex \lineskiplimit0pt
                     \hbox{\scriptsize.}\hbox{\scriptsize.}}}%
                     =}
\newcommand*{\eqdef}{=\mathrel{\vcenter{\baselineskip0.5ex \lineskiplimit0pt
                     \hbox{\scriptsize.}\hbox{\scriptsize.}}}%
                     }

\newcommand{\en}{\I}

\newcommand{\Hcal}{\mathcal H}

\renewcommand{\phi}{\varphi}

\let\Re\relax
\DeclareMathOperator{\Re}{Re}
\DeclareMathOperator{\re}{Re}
\let\Im\relax
\DeclareMathOperator{\Im}{Im}

\DeclareMathOperator{\inter}{int}
\DeclareMathOperator{\Img}{Im}

\DeclareMathOperator{\capp}{cap}
\DeclareMathOperator{\Real}{Re}
\DeclareMathOperator{\supp}{supp}

\DeclareMathOperator{\diam}{diam}
\DeclareMathOperator{\dist}{dist}

\DeclareMathOperator{\order}{ord}

\DeclareMathOperator{\U}{U}
\DeclareMathOperator{\D}{D}
\DeclareMathOperator{\I}{I}
\DeclareMathOperator{\CT}{C}

\hypersetup{
    colorlinks,
    citecolor=black,
    filecolor=black,
    linkcolor=black,
    urlcolor=black
}

\numberwithin{equation}{section}

\usepackage[nottoc]{tocbibind} 

\title{The Pólya-Tchebotarev problem with semiclassical external fields}

\usepackage{subcaption}

\author{Victor Alves\thanks{Corresponding author} \and Guilherme L.~F.~Silva}

\date{}

\makeatletter
\def\keywords{\xdef\@thefnmark{}\@footnotetext}
\makeatother

\makeatletter
\def\address{\xdef\@thefnmark{}\@footnotetext}
\makeatother

\begin{document}

\maketitle

\address{Instituto de Ci\^encias Matem\'aticas e de Computa\c{c}\~ao, Universidade de S\~ao Paulo, S\~ao Carlos, SP, 13566-590, Brazil, \texttt{victorjulio@usp.br, silvag@icmc.usp.br}}


\abstract{
The classical Pólya-Tchebotarev problem, commonly stated as a max-min logarithmic energy problem, asks for finding a compact of minimal capacity in the complex plane which connects a prescribed collection of fixed points. Variants of this problem have found ramifications and applications in the theory of non-hermitian orthogonal polynomials, random matrices, approximation theory, among others. 

Here we consider an extension of this classical problem, including a semiclassical external field, and enforcing finitely many prescribed collections of points to be connected, possibly also to infinity. 

Our method is based on Rakhmanov's approach to max-min problems in logarithmic potential theory, utilizes the developed machinery by Martínez-Finkelshtein and Rakhmanov on critical measures, and extends the development of Kuijlaars and the second named author from the context of polynomial external fields to the semiclassical case considered here.}

\tableofcontents

\section{Introduction}

One of the natural ways of measuring sizes of sets in the complex plane is through their capacity, which is roughly the total amount of charge a set can hold while keeping its energy. In the context of Harmonic Analysis, the solution of the Dirichlet Problem on a domain is tied to the capacity of the boundary of this domain. In Approximation Theory, convergence of rational approximants can be ensured except for sets of zero capacity. In Geometric Function Theory, the capacity of a compact domain is connected with the uniformization map of it. In orthogonal polynomials' theory, their zero distributions are computed through sets of minimal capacity. These and many other connections are discussed for instance in \cite{Stahl1986, Rakhmanov2012, Martinez-FinkelshteinRakhmanov2016}.

A classical problem in Geometric Function Theory, known as the Pólya-Tchebotarev Problem, asks for finding a compact and connected set of minimal capacity that contains a given finite collection of points $\Ccal\subset\C$ \cite{Polya29, OrtegaCerda2010}. In terms of the closely related notion of logarithmic energy, the Pólya-Tchebotarev Problem may be phrased as the question of finding a compact and connected set $K_*\subset \C$ containing $\Ccal$, and for which the max-min (or sup-inf) problem
\begin{equation}\label{eq:maxminChebotarev}
\sup_{\substack{\Ccal\subset K \\ K \text{ compact} \\ \text{and connected}}}\inf_{ \substack{\supp\mu\subset K} } \iint \log \frac{1}{|x-y|}\dd\mu(x)\dd\mu(y)
\end{equation}
is attained precisely at $K=K_*$, where in the above and in what follows $\mu$ always denotes a probability measure.

	The original form of the problem was presented by Pólya in 1929 \cite{Polya29}, and solved by Lavrentieff \cite{Lavrentiev30, Lavrentiev34}, Gr\"{o}tzch \cite{Grotzsch30} and Goluzin \cite{Goluzin46} in the following years.  Much later, motivated by questions on Padé approximants and rational approximation \cite{Stahl1997, Stahl1978, NuttallSingh1977}, Stahl established a more general version of Pólya-Tchebotarev problem \cite{Stahl1985a, Stahl1985b}, in that it weakens the finiteness condition on $\Ccal$, allowing it to be an arbitrary compact set with zero capacity.

Stahl's ideas were rather influential in particular in the theory of orthogonal polynomials. In his developments, he also connected sets of minimal capacity with non-hermitian orthogonal polynomials, showing that in the large degree limit the zeros of such polynomials are described by sets of minimal capacity \cite{Stahl1986}, see also \cite{GoncharRakhmanov1987}. Starting in the early 2000s, such ideas and results also found their way with applications to other areas of mathematics such as numerical analysis \cite{DeanoHuybrechsKuijlaars2010, HuybrechsKuijlaarsLejon2014, CelsusSilva2020, Deano2014, MartinezFinkelshteinRakhmanovSuetin2011, Yattselev2018, Atiaetal2014}, random matrix theory \cite{AptekarevLysovTulyakov2011, BarhoumiBleherDeanoYattselev2022, BleherGharakhlooMclaughlin, bertola2011}, integrable systems \cite{BuckinghamMiller22, BertolaBothner2015, KamvissisRakhmanov2005, ChouikhiThabet2019}, among others. 

In many of these applications we just mentioned, the corresponding extremal problems that arise often are analogues of the original Pólya-Tchebotarev Problem, but now under the setting of {\it weighted} capacity, which is the main focus of our work here. We solve a weighted extension of the original Pólya-Tchebotarev Problem \eqref{eq:maxminChebotarev}, obtained including an external field in the energy functional: for a given finite set $\Ccal\subset \C$ and a suitable function $\phi$ defined over $\C$, find a compact and connected set $K_*\subset \C$ containing $\Ccal$, and for which the max-min problem
\begin{equation}\label{eq:maxminChebotarevext}
\sup_{\substack{\Ccal\subset K \\ K \text{ compact} \\ \text{and connected}}}\inf_{ \substack{\supp\mu\subset K} } \left( \iint \log \frac{1}{|x-y|}\dd\mu(x)\dd\mu(y)+ 2\int \phi(x) \dd\mu(x)\right)
\end{equation}
is attained at $K=K_*$.

We consider this problem for a family of semiclassical external fields with, and in fact more general class of sets, which appear naturally for instance in the context of non-hermitian orthogonality. To give a glimpse of our results in a simpler situation, let us for now assume that $\phi=\Re \Phi$ is the real part of a monic, complex-valued polynomial $\Phi$, with $\deg \Phi$ being a positive even number.

A monic polynomial $p_n$ of degree $n$ is said to be the {\it non-hermitian orthogonal polynomial} of degree $n$ with respect to the weight $\ee^{-n \Phi}$ if it satisfies the orthogonality condition
\begin{equation}\label{eq:nonhermitianOP}
\int_\Gamma p_n(z)z^k \ee^{-n\Phi(z)} \dd z=0,\quad k=0,\hdots, n-1.
\end{equation}
Assume that $\Gamma\in \Tcal$ is a representative of one of the following classes of contours:
\begin{itemize}
\item $\Tcal=\Tcal_{\text{Jac}}$: the class of contours connecting the points in $\Ccal=\{-1,1\}$.
\item $\Tcal=\Tcal_{\text{Lag}}$: the class of contours connecting the point in $\Ccal=\{0\}$ to $+\infty$.
\item $\Tcal=\Tcal_{\text{Herm}}$: the class of contours connecting $-\infty$ to $+\infty$.
\end{itemize}

When $\Phi$ is a constant/linear/quadratic real-valued polynomial and $\Tcal=\Tcal_{\text{Jac}}/\Tcal_{\text{Lag}}/\Tcal_{\text{Herm}}$, the polynomial $p_n$ reduces to (particular cases of rescaled) Jacobi/Laguerre/Hermite polynomial, respectively. For general complex-valued polynomial $\Phi$, the polynomial $p_n$, if it exists, is independent of the contour of orthogonality $\Gamma\in \Tcal$ chosen, since the orthogonality condition \eqref{eq:nonhermitianOP} remains unchanged if we deform $\Gamma$ to a new contour $\widetilde \Gamma$ within the same class. 

In each of the aforementioned cases, the flexibility in the choice of contours $\Gamma\in \Tcal$ has impact in the behavior of zeros of the corresponding OPs. As we learn from the work of Gonchar and Rakhmanov \cite{GoncharRakhmanov1987}, greatly inspired by the earlier works of Stahl \cite{Stahl1985a, Stahl1985b}, if there exists a contour $\Gamma_*\in \Tcal$ for which the max-min problem \eqref{eq:maxminChebotarevext} (when the supremum is taken over $K\in \Tcal$) is attained, then the zeros of $p_n$ should accumulate, in the large degree limit, to the contour $\Gamma_*$. Furthermore, these zeros must also be distributed according to the equilibrium measure of $\Gamma_*$, that is, the probability measure for which the infimum in \eqref{eq:maxminChebotarevext} over $\Gamma_*$ is attained. Moreover, $\Gamma_*$ is uniquely determined up to a zero capacity set.

Even though the existence of the max-min contour $\Gamma_*$ in the cases we just mentioned was conjectured in the community already in the 1990s, the rigorous proof of its existence in the case $\Tcal=\Tcal_{\text{Herm}}$ was established much later, by Kuijlaars and one of the authors \cite{Kuijlaars2015}, based on a strategy outlined by Rakhmanov \cite{Rakhmanov2012}. And in the other two cases $\Tcal=\Tcal_{\text{Lag}}, \Tcal=\Tcal_{\text{Jac}}$, the existence of $\Gamma_*$, to this date, is only known for some concrete simple choices of the potential $\Phi$, and not as a true max-min solution but rather as a contour that attracts the zeros of the orthogonal polynomials, found explicitly with ad hoc methods.

Our main result establishes the existence of the max-min contour for each of the classes $\Tcal=\Tcal_{\text{Herm}},\Tcal = \Tcal_{\text{Lag}}, \Tcal=\Tcal_{\text{Jac}}$ as above, also for classes of contours obtained from unions of elements of those, and much more, allowing also for semiclassical external fields. Thus, with our work, we place many of the aforementioned works under the same framework, and showcase how one can leverage some techniques that have been developed and explored in the last years to more general versions of the Pólya-Tchebotarev Problem.

\section{Statement of results}

	
Before proceeding to the rigorous statement of our main result, we introduce some necessary mathematical ingredients and notions.

\subsection{Necessary concepts from potential theory}\label{subsec:PotTheory}

	Given a finite Borel measure $\mu$ on $\C$, the functions
	\begin{equation}\label{deff:logpotCauchytransf}
		\U^\mu(z) \defeq \int \log \frac{1}{|x-z|} \, \dd \mu(x)\qquad \text{and}\qquad \CT^{\mu}(z)\defeq \int\frac{\dd\mu(x)}{x-z}, \quad z\in \C\setminus \supp\mu,
	\end{equation}
	are called the \emph{logarithmic potential} and the {\it Cauchy transform} of $\mu$, respectively. The logarithmic potential $\U^\mu$ is a harmonic function on $\C\setminus \supp\mu$, and under the additional condition
	\begin{equation}\label{eq:GrowthCond}
		\int \log (1+|x|) \, \dd \mu(x) < \infty,
	\end{equation}
it takes values in $(-\infty,+\infty]$, and extends to a superharmonic function on the whole plane $\C$. The Cauchy transform $\CT^\mu$ is an analytic function on $\C\setminus \supp\mu$. It is the convolution of the measure $\mu$ with the Cauchy kernel, so it is locally integrable with respect to the planar Lebesgue measure, and therefore it extends Lebesgue almost everywhere to a finite function on $\C$. With $\partial_z\defeq \frac{1}{2}(\partial_x-\ii\partial y)$, these functions are related by
$$
2\partial_z\U^\mu(z)=\CT^\mu(z),\quad z\in \C\setminus \supp\mu,
$$
relation which also extends to values on $z\in \supp\mu$ in the weak sense.

Still for measures satisfying \eqref{eq:GrowthCond}, their logarithmic energy 
$$
\I(\mu) \defeq
\iint \log\frac{1}{|x-y|}\dd\mu(x)\,\dd\mu(y)
=\int \U^\mu(x) \, \dd \mu(x)
$$
is well defined with values on $(-\infty,+\infty]$,

For a set $F \subset \C$ which is closed, we denote by $\Mcal_1(F)$  the set of Borel probability measures supported on $F$ that satisfy \eqref{eq:GrowthCond} and have finite logarithmic energy. Given a function $\phi$ defined on $F$, we denote by $\Mcal_1^\phi(F)$ the subset of measures $\mu\in \Mcal_1(F)$ that also satisfy $\phi \in L^1(\mu)$. Thus, for $\mu \in \Mcal_1^ \phi(F)$ the \emph{(weighted logarithmic) energy} 
	$$
		\I^\phi(\mu) \defeq \I(\mu) + 2\int \phi(x) \, \dd \mu(x)
	$$	
is well defined, and takes values in $\R$.

The weighted logarithmic energy minimization problem asks for finding $\mu_F = \mu(\phi,F) \in \Mcal_1^\phi(F)$ such that
	$$
		\I^\phi(\mu_F) = \inf_{\mu \in \Mcal_1^\phi(F)} \I^\phi(\mu) \eqdef \I^\phi(F).
	$$
	The measure $\mu_F$, when it exists, is called the \emph{equilibrium measure of $F$ in the external field $\phi$} or simply the \emph{equilibrium measure of $F$} when $\phi$ is understood. 
	
Let $\Gamma$ be a closed set in $\C$ that satisfies the growth condition
	\begin{equation}\label{eq:GrowthCondGamma}
		\lim_{\substack{z \to \infty \\ z \in \Gamma}} \left[ \phi(z) - \log |z| \right] = +\infty,
	\end{equation}
and let $\Zcal$ be the set of discontinuities of $\phi$. For us the $\phi$ shall always be such that $\Zcal$ is finite.
Then the equilibrium measure $\mu(\phi,\Gamma)$ uniquely exists  \cite{Saff1997}, provided that $\Gamma$ has positive capacity, and is characterized by the {\it Euler-Lagrange variational conditions}: it is the unique measure in $\Mcal_1^\phi(\Gamma)$ for which there exists a constant $\ell$ satisfying
	\begin{equation}\label{eq:EulerLagrange}
		\U^\mu(z) +  \phi(z) = \ell, \quad z \in \supp \mu \setminus \Zcal, \qquad \text{and}\qquad
		\U^\mu(z) +  \phi(z) \geq \ell, \quad z \in \Gamma.
	\end{equation}	
In general, conditions \eqref{eq:EulerLagrange} should be interpreted in the quasi-everywhere sense, but whenever $\Zcal$ is finite and $\Gamma$ is a finite union of continua (compact sets with more than one point), they hold as stated.

When $\phi \equiv 0$ and $F$ is a compact set with positive capacity, the class of measures $\Mcal_1(F)=\Mcal_1^{\phi\equiv 0}(F)$ is non-empty and the problem just stated is well posed. In this case  $\mu_F$ uniquely exists and is called the \emph{Robin measure} of $F$, the value $\I(F)$ is called the \emph{Robin constant} of $F$, and $\capp(F) = \ee^{-\I(F)}$ is the (logarithmic) \emph{capacity} of $F$. 

Fix a finite set $\Ccal\subset \C$. Still in this context of $\varphi\equiv 0$, the previously mentioned classical Pólya-Tchebotarev Problem \eqref{eq:maxminChebotarev} was originally stated in the equivalent formulation of finding a compact and connected set $K_*$ with $\Ccal\subset K_*$ and for which
$$
\capp(K_*)=\inf_{\substack{\Ccal\subset K  \\ K\text{ compact }\\ \text{and connected}}} \capp(K).
$$
		
As mentioned, in the work we consider here we are interested in an extension of the classical Pólya-Tchebotarev Problem \eqref{eq:maxminChebotarev}, but now for the weighted energy $\I^\varphi$. We assume from now on that $\phi$ is the real part of the sum of a rational function and a logarithm part. Without loss of generality, we place one of the poles of the rational function at $\infty$. Later on, when posing our extremal problem we will allow contours extending to $\infty$ along different sectors, and to avoid degenerate situations we assume that the pole at $\infty$ is of order at least $1$. In summary, throughout this paper we always assume that 
	\begin{equation}\label{eq:PhiDef}
	\phi \defeq \Real \Phi, \quad \text{with}\quad 
	\Phi(z) \defeq \frac{P(z)}{Q(z)} + L(z),
	\end{equation}
where $P,Q$ are polynomials and $L$ is a logarithmic component, which are assumed to satisfy the following conditions. First of all, we assume
$$
\gcd(P,Q) = 1, \quad
	 N \defeq \deg P - \deg Q \geq 1.
$$
We normalize $Q$ to be monic, and write
\begin{equation}\label{eq:formQ}
Q(z) \defeq \prod_{j=1}^{n_Q} (z-z_j)^{m_j}, \quad m_j>0,\quad  j=1,\ldots, n_Q,
\end{equation}
so that $z_j$ is a pole of order $m_j$ of $\Phi$. For the logarithmic component, we assume
\begin{equation}\label{eq:formL}
	L(z) \defeq \log \prod_{j=1}^{n_L}  (z-w_j)^{\rho_j} = \sum_{j=1}^{n_L} \rho_j \log (z-w_j), \quad \rho_j \in \R \setminus \{0\},\; j=1,\ldots,n_L.
\end{equation}
Lastly, we assume that all $z_j$ and $w_j$ are pairwise distinct.

The condition $N\geq 1$ ensures that $\infty$ is a pole of the rational part of $\Phi$, and also that condition \eqref{eq:GrowthCondGamma} holds true along at least one nontrivial sector at $\infty$. 


	Note that the weight function $w(z) = \ee^{- \Phi(z)}$ is a \emph{semiclassical} weight, i.e., it satisfies
	\[
		\frac{w'}{w} = \Phi',
	\]	
	and $\Phi'$ is a rational function, which motivates us to call $\Phi$ a {\it semiclassical external field}.
	
We also write
\begin{equation}\label{eq:defZ}
	\Zcal \defeq \{ z_1,\ldots, z_{n_Q}, w_1, \ldots, w_{n_L}\}
\end{equation}
that is, $\Zcal$ is the set of singularities of $\Phi$.
	
The class of external fields we consider is now introduced, and we move on to the introduction of the class of contours of interest.

	\subsection{The class of admissible contours}	  

For us, we will consider the weighted version of \eqref{eq:maxminChebotarev} on a suitable large class of sets $\Tcal$ consisting of unions of contours. 

In the classical problem 	\eqref{eq:maxminChebotarev}, the class of sets consist of compact and connected sets containing $\Ccal$. As shown by Stahl \cite{Stahl1985a, Stahl1985b}, the set of minimal capacity $K_*$ is in fact a union of analytic arcs, and consequently the max-min problem \eqref{eq:maxminChebotarev} could as well be posed on the more restrictive class of sets which are unions of contours containing $\Ccal$, while still preserving its minimizer $K_*$. 

In our extension, we will allow sets $\Gamma\in \Tcal$ which are unions of contours with at least one of three distinguished features when compared with the classical Pólya-Tchebotarev setup. The first feature is enforcing that prescribed subcollections of points in $\Ccal$ must belong to the same connected components of a given $\Gamma\in \Tcal$. The second feature is enforcing that prescribed subcollections of points of $\Ccal$ to belong to the same connected components of $\Gamma\in \Tcal$, and each of these connected components must extend to $\infty$ along prescribed angular sectors. The third feature is enforcing that each $\Gamma\in \Tcal$ must have connected components stretching to $\infty$ along prescribed angular sectors. The canonical examples of connected components coming from each of these features are provided by contours in $\Tcal_{\text{Jac}}, \Tcal_{\text{Lag}}$ and $\Tcal_{\text{Herm}}$, respectively.

We now introduce appropriate notions to book-keep these properties in a careful and detailed manner. The notions we use in what follows are inspired by \cite{Kuijlaars2015}.

For what follows, denote
	$$
		D_R(z_0) \defeq \{z \in \C; |z-z_0| < R \}, \quad D_R \defeq D_R(0).
	$$	 

Consider a set of fixed points $\Ccal$, writing
	\begin{equation}\label{eq:deffC}
\Ccal = \{c_1,\ldots,c_{n_\Ccal}\} \quad \text{and}\quad C(z) \defeq \prod_{j=1}^{n_\Ccal} (z-c_j).
	\end{equation}
The points in $\Ccal$ will play the role of fixed points in the Pólya-Tchebotarev problem. If $\Ccal = \emptyset$ set $C(z) \defeq 1$.

We also throughout the paper assume that
\begin{equation}\label{eq:polesfixedpts}
\Ccal\cap \Zcal =\emptyset,
\end{equation}
that is, points in $\Ccal$ are not poles of $\Phi$.

Next, we need to discuss about sectors at $\infty$ that may be connected to each other.

	For $z = R\ee^{\ii \theta}$ and $\alpha$ the leading coefficient of the polynomial $P$, the behavior
	$$
		\phi(z) = R^{N} |\alpha|\cos\left( N \theta + \arg(\alpha) \right) \left[1 + o(1)  \right], \quad z \to +\infty
	$$  
is of immediate verification from \eqref{eq:PhiDef}. In particular, this shows that there are exactly $N$ sectors $S_1,\ldots,S_N$ such that $\phi(z) \to +\infty$ as $z \to \infty$ on $S_j$, which are given by
\begin{equation}\label{eq:deffsecSj}
		S_j \defeq \left\{ z \in \C; | \arg(z) - \theta_j| < \frac{\pi}{2N} \right\}, 
	\quad \theta_j \eqdef - \frac{\arg \alpha}{N} + \frac{2 \pi (j-1)}{N}, 
	\quad 1 \leq j \leq N.
\end{equation}
We call $S_1,\hdots, S_N$ the {\it admissible sectors for the external field $\phi$}, or simply {\it admissible sectors}, as for the rest of the paper the potential $\phi=\re\Phi$ is fixed and always of the form \eqref{eq:PhiDef},

	We say that a set $F \subset \C$ \emph{stretches to infinity} in the sector $S_j$ if there exists $\epsilon > 0$ and $r_0 > 0$ such that for every $r > r_0$, there is $z \in F$ with
	$$
		|z| = r, \quad |\arg z - \theta_j| < \frac{\pi}{2N} - \epsilon.
	$$
	In particular, this implies that $\phi(z) \to +\infty$ as $z \to \infty$ within $S_j \cap F$.

Let $\Theta = \{\theta_1 < \cdots < \theta_N\}$ be the set of the central angles of the admissible sectors as introduced in \eqref{eq:deffsecSj}. We shall use $\theta_j$ to represent the sector $S_j$. We talk about partitions $\Pcal(\Theta)$ of $\Theta$, which will be used when encoding the connection of different sectors at $\infty$ in our max-min problem. A partition $\Pcal(\Theta)$ of $\Theta$ introduces an equivalence relation $\sim$ in the usual way: $\theta_j \sim \theta_k$ if both $\theta_j$ and $\theta_k$ belong to the same set in $\Pcal(\Theta)$. We denote by $\Pcal_0(\Theta)$ the subset of $\Pcal(\Theta)$ obtained by removing all classes consisting of only one element.
	
A partition $\Pcal(\Theta)$ of $\Theta$  is said to be non-crossing if its associated equivalence relation $\sim$ satisfies the following property: if $\theta_j \sim \theta_{j'}$ and $\theta_k \sim \theta_{k'}$ with $j<k<j'<k'$, then necessarily $\theta_j\sim \theta_k$.  Later on, we also talk about non-crossing partitions of subsets of $\Theta$.

The name non-crossing is better motivated when we represent this equivalence relation on the unit circle as follows. Represent each element $\theta$ of $\Theta$ on the unit circle as the point $(\cos(\theta), \sin(\theta))$ and connect consecutive elements of an equivalence class of the partition by a contour inside the circle, also connecting the first and last elements of the same class. A partition is then non-crossing if there exists a choice of these contours in such a way that different contours can only intersect if they connect elements of the same class. We refer the reader to Figures \ref{fig:NonCrossingPartition} and \ref{fig:CrossingPartition} for examples. 
	\begin{figure}[t]
		\begin{subfigure}{.5\textwidth}
  \centering
  \includegraphics[scale=.17]{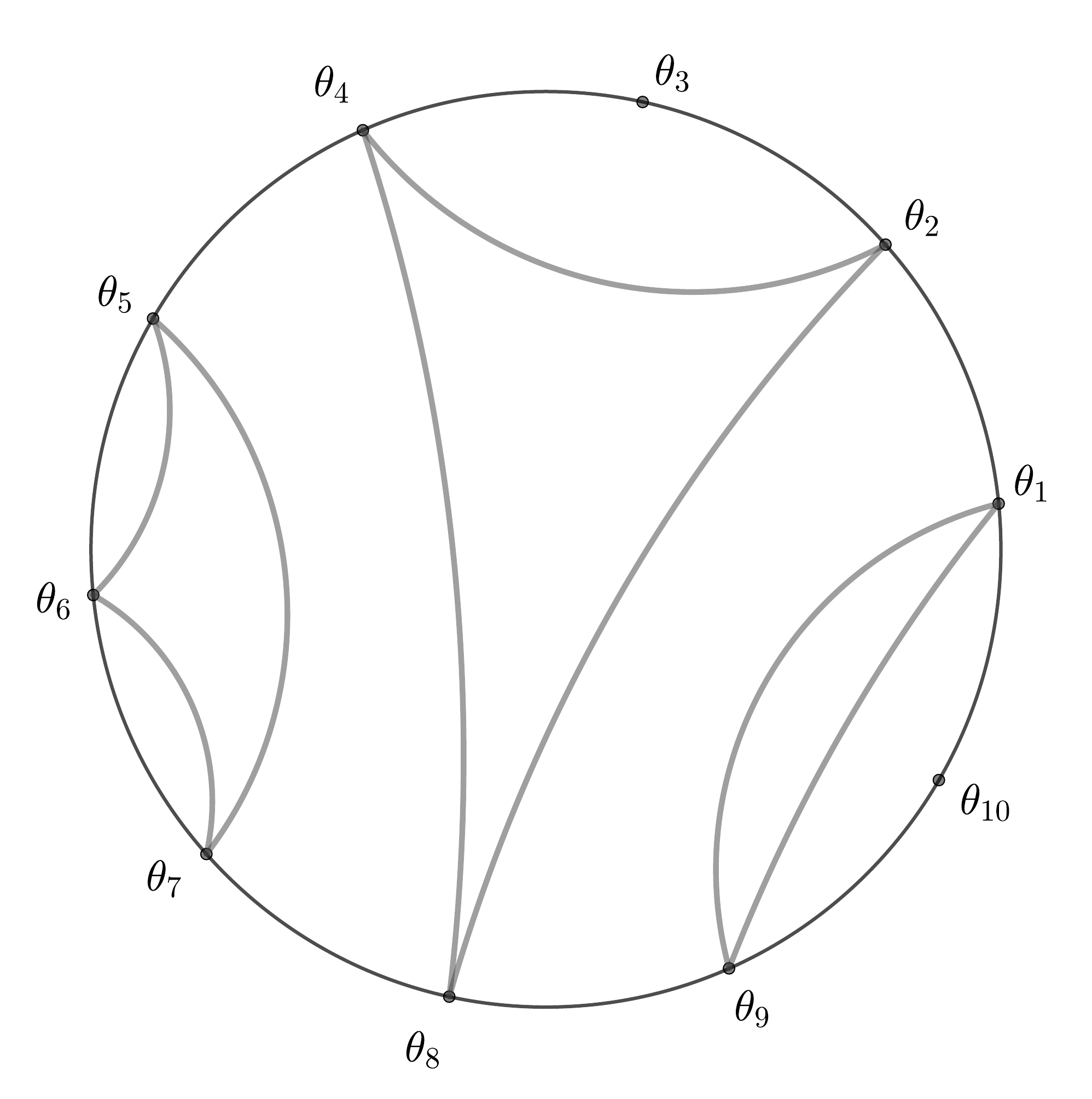}
 \caption{A non-crossing partition}
  \label{fig:NonCrossingPartition}
\end{subfigure}%
\begin{subfigure}{.5\textwidth}
  \centering
  \includegraphics[scale=.17]{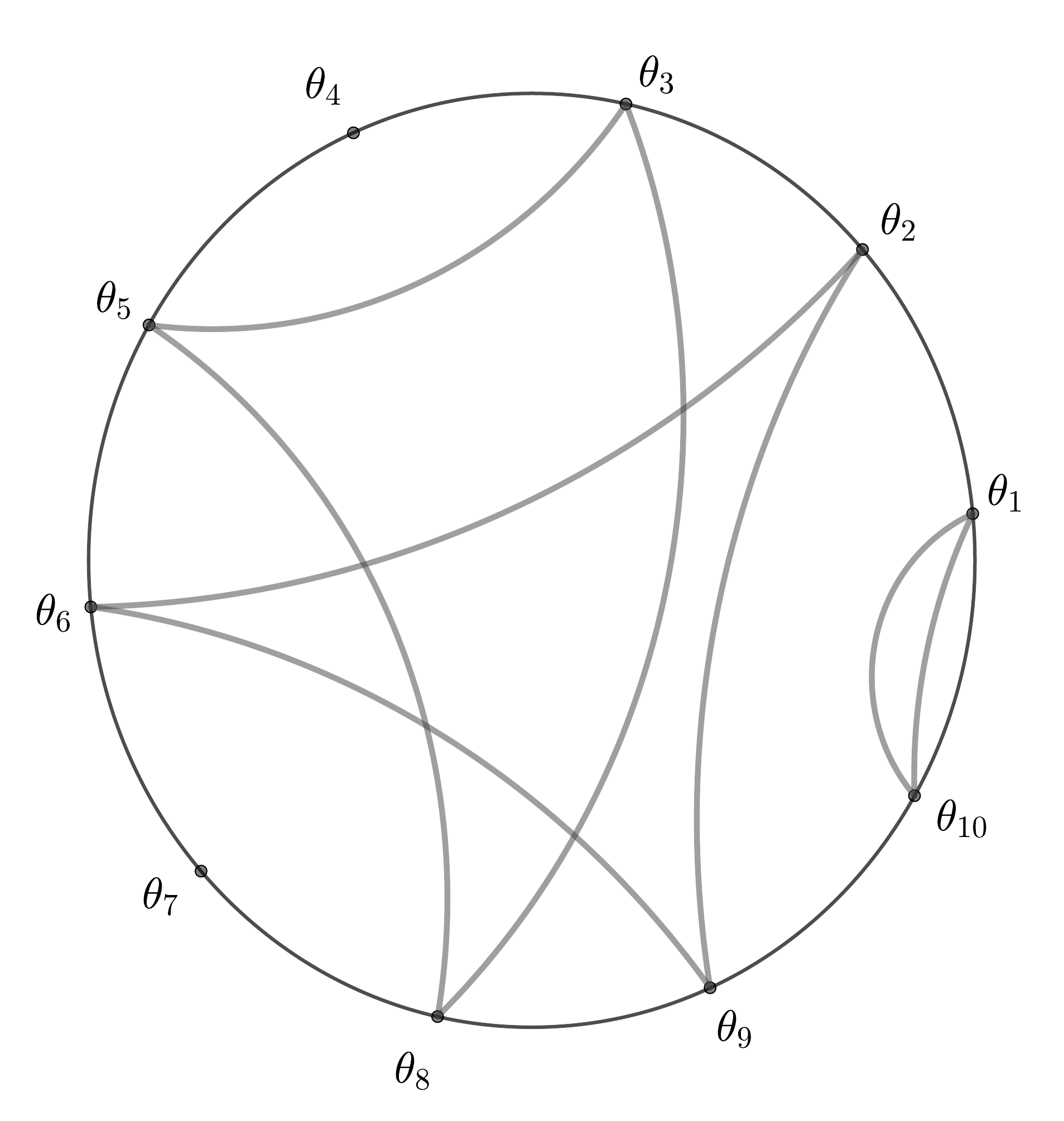}
  \caption{A crossing partition}
  \label{fig:CrossingPartition}
\end{subfigure}
\caption{Examples of crossing and non-crossing partitions for a set $\Theta$ of $10$ elements. Any choice of contours that represent the crossing partition have an intersection between at least two contours, and there is no contour representation without this property.}
\label{fig:Partitions}
\end{figure}

In a similar manner, we will encode the connections of finite points through a partition $\Pcal(\Ccal)$ of $\Ccal$. As said earlier, we want to define classes of sets $\Tcal$ for our Pólya-Tchebotarev problem, which involves connecting finite points and $\infty$.  In what follows, we talk about {\it connecting} points $c_i,c_j \in \Ccal$ through a set $\Gamma$, meaning that $c_i$ and $c_j$ belong to the same connected component of $\Gamma$. Likewise, we talk about connecting a point $c_i$ to an admissible sector $S_k$, meaning that $c_i$ belongs to a connected component of $\Gamma$ that stretches to $\infty$ along $S_k$. Finally, we also talk about connecting different sectors $S_k$ and $S_l$, meaning that there is a connected component of $\Gamma$ that stretches to $\infty$ in both sectors $S_k$ and $S_l$. 
	
The next step is to book-keep all such possible connections in a convenient notation, but before moving forward it is instructive to discuss the coming notions with an example.
	
		\begin{example}\label{example1}
		Fix a set of distinct points $\Ccal=\{c_1,\hdots, c_6\}$. Say that our external field determines exactly eight admissible sectors, the most trivial case being $\phi(z) = \Real z^8$. We define now a class of sets $\Tcal$ with the following features.
		
		First of all, a set $\Gamma\in \Tcal$ must connect points in $\Ccal$ with the following rules:
		\begin{itemize}
			\item The point $c_1$ is not necessarily connected to any other finite point.
			\item The points $c_2, c_3$ and $c_4$ are necessarily connected.
			\item The points $c_5$ and $c_6$ are necessarily connected.
		\end{itemize}
		
		In terms of the previously defined notation, these rules translate to fixing the partition 
		\[
		\Pcal(\Ccal) = \{ \{c_1\}, \{c_2,c_3,c_4\}, \{c_5,c_6 \} \}
		\] 
		which codify connections above - two points of $\Ccal$ are connected to each other if, and only if, they are elements of the same equivalence class defined by $\Pcal(\Ccal)$. See Figure \ref{fig:AdmContourExample_Img1} for an example of a set that satisfy above connections.
		\begin{figure}[t]
		\centering	\includegraphics[scale=0.4]{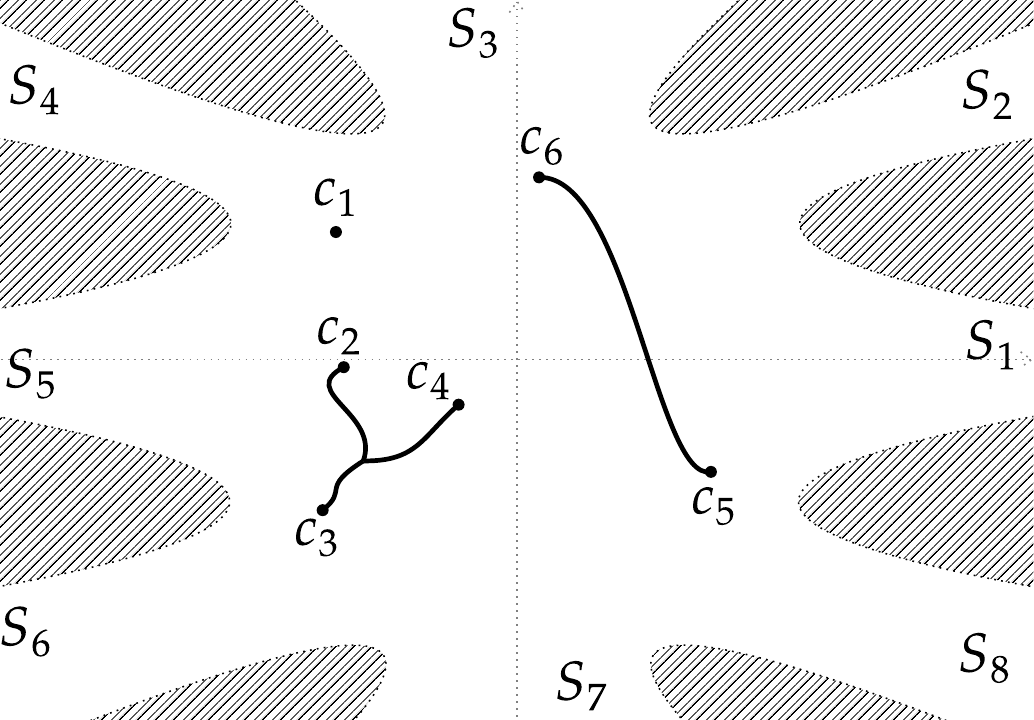}
		\caption{An example of a set $F$ with the connections between the fixed points.}
		\label{fig:AdmContourExample_Img1}
		\end{figure}
		
		Next, we determine the rules of connections of points in $\Ccal$ to $\infty$. Consider the following:
		\begin{itemize}
			\item The class $\{c_1\}$ must be connected to $\infty$ through sectors $S_3$ and $S_7$.
			\item The class $\{c_2,c_3,c_4\}$ must be connected to $\infty$ through sectors $S_1$ and $S_6$. 
			\item The class $\{ c_5, c_6\}$ is not necessarily connected to infinity.
		\end{itemize}
		We would like to encode these connections preserving the codification given by $\Pcal(\Ccal)$. We do this through a function $\Psi$, defined on $\Pcal(\Ccal)$, with the mapping rules
$$
\Psi(\{c_1\}) = \{ \theta_3, \theta_7 \},\quad \Psi(\{c_2,c_3,c_4\}) = \{\theta_1, \theta_6\}, \quad \text{and}\quad \Psi(\{c_5,c_6 \}) =  \emptyset,$$
so that in general $\Psi:\Pcal(\Ccal)\to 2^\Theta$.

	\begin{figure}[t]
		\centering	\includegraphics[scale=0.4]{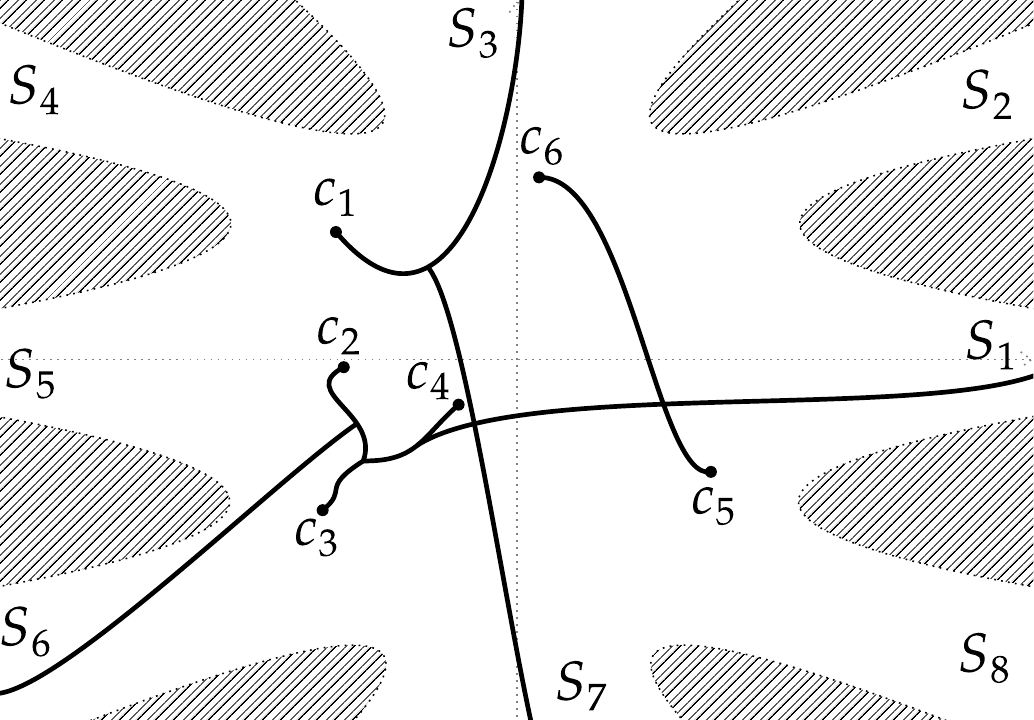}
		\caption{An example of a set $F$ with the connections between the fixed points, and connections to infinity through admissible sectors. Note that no matter how you try to connect following the described rules, a connection between $c_1$ and the points $c_2,c_3$ and $c_4$ is established.}
		\label{fig:AdmContourExample_Img2}
		\end{figure}
		
		Note that the connections made above imply that the point $c_1$ is connected to $c_2, c_3, c_4$ (see Figure \ref{fig:AdmContourExample_Img2}). This creates a problem in the construction, in the sense that partition $\Pcal(\Ccal)$ is not properly describing the connections between fixed points. To evade this mishap, we shall assume that 
		\begin{equation}\label{eq:Example1Cond1}
			\Pcal_\Psi \defeq \{ \Psi(A); A \in \Pcal(\Ccal)\} \text{ is a non-crossing partition of }\bigcup_{A \in \Pcal(\Ccal)} \Psi(A).
		\end{equation}
		 Initially this may seem to restrict our construction, but that is not really the case, one can restart the construction taking these considerations in mind. For example, with the choices made above, since $c_1$ will always be connected to $c_2$, $c_3$ and $c_4$, the correct associated partition is 
		\[
			\Pcal(\Ccal) = \{ \{c_1, c_2,c_3,c_4\}, \{c_5,c_6 \} \}
		\]
		with 
		\[
			\Psi(\{c_1,c_2,c_3,c_4\}) = \{ \theta_1, \theta_3, \theta_6, \theta_7 \}, \quad \Psi(\{c_5,c_6\}) = \emptyset. 
		\]
		Following the process above, one always finds a partition $\Pcal'(\Ccal)$ and a function $\Psi'$ defined on $\Pcal'(\Ccal)$ that satisfies condition \eqref{eq:Example1Cond1} and properly describes the connections.

		Finally, we would like to describe connections between different admissible sectors, not necessarily passing through fixed points, and we need to be careful to preserve the structure given by  $\Pcal(\Ccal)$ and $\Psi$. This will translate in another condition to be assumed on the construction. For example, assume that
		\begin{itemize}
			\item Sectors $S_1$ and $S_5$ are connected.
		\end{itemize}
		Then the class $\{c_1, c_2, c_3, c_4\}$ is connected to $S_5$ (through infinity), which is an information not contained in $\Psi$. Consider that
		\begin{itemize}
			\item Sectors $S_2$ and $S_5$ are connected.
		\end{itemize}
		While information above does not concern any of the admissible sectors already present on function $\Psi$, it implies that $\{c_1, c_2, c_3, c_4\}$ is connected to $S_2$ and $S_5$, which once again discredits $\Psi$. Finally, some redundancy can come just by the type of connections we would like to make, for example
		\begin{itemize}
			\item Sectors $S_2$ and $S_5$ are connected.
			\item Sectors $S_4$ and $S_8$ are connected.
		\end{itemize}
		This would imply that sectors $S_2, S_4, S_5$ and $S_8$ are connected. To properly handle this in a simple way we start by describing connections above as a partition $\Pcal(\Theta)$ of $\Theta$, where singleton sets corresponds to sectors that we do not connect, and equivalence classes represent connections, and assume that
		\begin{equation}\label{eq:Example1Cond2}
			\Pcal(\Theta) \text{ is a non-crossing partition of } \Theta \text{ and } \Pcal_\Psi \subset \Pcal(\Theta).
		\end{equation}
		For example, we could say that
		\begin{itemize}
			\item Sectors $S_4$ and $S_5$ are connected.
		\end{itemize}
		and thus the partition takes the form
		\[
			\Pcal(\Theta) = \{ \{\theta_1, \theta_3, \theta_6, \theta_7\}, \{\theta_2\}, \{\theta_4,\theta_5\}, \{\theta_8\}\}
		\]
		and satisfies \eqref{eq:Example1Cond2}. Once again, when representing connections by a partition $\Pcal(\Theta)$ that initially does not seem to satisfy $\eqref{eq:Example1Cond2}$,  one can rework from the beginning and find corresponding partitions $\Pcal'(\Ccal)$, $\Psi'$ and $\Pcal'(\Theta)$  that properly describe the connections being made and satisfy conditions \eqref{eq:Example1Cond1} and \eqref{eq:Example1Cond2}).
		
		Assuming \eqref{eq:Example1Cond1} and \eqref{eq:Example1Cond2}, we have the correct levels of codification on the sets $\Pcal(\Ccal)$, $\Pcal_\Psi$ and $\Pcal(\Theta)$. For example, if $\Ccal$ has $6$ elements and $\Theta$ has seven, putting 
		\begin{equation}\label{ex:partitions}
			\begin{split}
			\Pcal(\Ccal) &= \{\{c_1\}, \{c_2,c_3,c_4\}, \{ c_5,c_6\}\} \\
			\Psi(\{c_1\}) &= \{\theta_4\}, \quad \Psi(\{c_2,c_3,c_4\}) =\emptyset ,\quad  \Psi(\{ c_5,c_6\}) = \{\theta_3\}  \\
			\Pcal(\Theta) &= \{\{\theta_1\}, \{\theta_2\}, \{\theta_3\} ,\{\theta_4\}, \{\theta_5, \theta_6\}, \{\theta_7\} \}
			\end{split}
		\end{equation}
		one readily sees that \eqref{eq:Example1Cond1} and \eqref{eq:Example1Cond2} are satisfied.
		Generically, sets of $\Tcal$ are as in Figure \ref{fig:Example}.
		
		\begin{figure}[ht!]
		\centering	\includegraphics[scale=0.4]{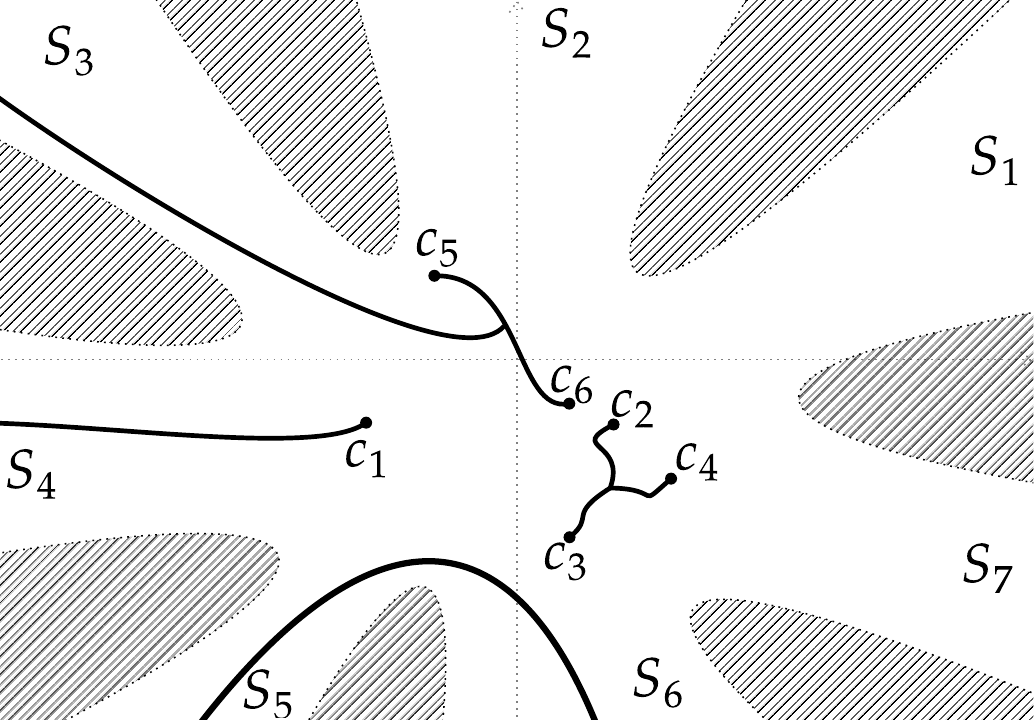}
		\caption{A generic element of $\Tcal$ for the codification given by \eqref{ex:partitions}.}
		\label{fig:Example}
		\end{figure}
	\end{example}

Hence, in the previous example we determined $\Tcal$ through the triple $(\Pcal(\Ccal),\Psi,\Pcal(\Theta_\infty))$.  We now proceed to the discussion needed for the general definition.
	
We book-keep subcollections of points of $\Ccal$ that will be necessarily connected by contours in our Pólya-Tchebotarev Problem by the classes of a given partition $\Pcal(\Ccal)$ of $\Ccal$, then two points are connected by a contour if, and only if, they are elements of the same partition. 

Next, we need to specify which points of $\Ccal$ will be enforced to connect to $\infty$, along prescribed sectors. To that end, we represent the connections of finite fixed points to $\infty$ by a fixed function $\Psi: \Pcal(\Ccal) \to 2^\Theta$, which identifies which points in $\Ccal$ are connected to $\infty$ through which sector in the sense that there is a contour connecting $c$ to infinity through sector $S_j$ if, and only if, $\theta_j$ is an element of $\Psi(A)$, where $A$ is the equivalence class of $c$ on $\Pcal(\Ccal)$. 

 We assume that
	\begin{align}
			& \Psi(\{w\}) \neq \emptyset \text{, for all singleton sets } \{w\} \in \Pcal(\Ccal),\qquad \text{and} \label{CondPsi1}\\
			& \Psi(A) \cap \Psi(B) \neq \emptyset \text{ implies } A = B. \label{CondPsi2}
	\end{align}	
Both conditions are placed to avoid redundancies. Condition \eqref{CondPsi1} means that if we have a fixed point not connected to any other fixed point, it must be connected to infinity through some admissible sector. We enforce this condition otherwise one could remove $w$ from the family of fixed points.  Condition \eqref{CondPsi2} means that we do not have different equivalence classes of finite points being enforced to connect to each other through infinity through the same sector, otherwise we should include then in the same class in $\Pcal(\Ccal)$ from the very beginning.

As the key property, we ask that
\begin{equation}\label{eq:CondPsi3}
\Pcal_\Psi\defeq \left\{ \Psi(A); A \in \Pcal(\Ccal) \right\} \text{ is a non-crossing partition of } \bigcup_{A \in \Pcal(\Ccal)} \Psi(A).
\end{equation}
This solidifies the codification nature of $\Pcal(\Ccal)$ (see Example \ref{example1}): by removing this condition one can create examples of $\Psi$ which enforces connections between fixed points in different classes of $\Pcal(\Ccal)$.

We have yet to describe connections between admissible sectors not necessarily passing through any prescribed point in $\Ccal$, and for that we need to talk about another notion. 

We say that a partition $\Pcal(\Theta)$ is {\it subordinated} to the function $\Psi$ if $\Pcal(\Theta)$ is a non-crossing partition of $\Theta$ with $\Pcal_\Psi\subset \Pcal(\Theta)$. The classes of $\Pcal_0(\Theta) \setminus \Pcal_\Psi$ codify connections between admissible sectors not necessarily containing a fixed point, where two angles $\theta_i$ and $\theta_j$ are in the same class if, and only if, there is a contour stretching to infinity on $S_j$ and $S_k$.
		
	With the previous notations we have the necessary objects to define the class of admissible sets.
	
	\begin{definition}\label{def:AdmissibleClass} Fix a potential $\phi=\re\Phi$ as in \eqref{eq:PhiDef} with corresponding set of admissible angles $\Theta$ as described in \eqref{eq:deffsecSj} {\it et seq.}, and a finite collection of points $\Ccal$ as in \eqref{eq:deffC}.
	
	Let $\Pcal(\Ccal)$ be a partition of $\Ccal$, $\Psi:\Pcal (\Ccal)\to 2^\Theta$ a function satisfying (\ref{CondPsi1}), (\ref{CondPsi2}) and \eqref{eq:CondPsi3}, and $\Pcal(\Theta)$ a partition of $\Theta$ subordinated to $\Psi$. Assume that $\Ccal \neq \emptyset$ or $\Pcal_0(\Theta) \neq \emptyset$. We call $(\Pcal(\Ccal), \Psi, \Pcal(\Theta))$ an {\it admissible triple}, and associate to it a collection of contours $\Tcal(\Pcal(\Ccal), \Psi, \Pcal(\Theta))$, referred to as {\it admissible family} as follows. Each set $\Gamma\in \Tcal(\Pcal(\Ccal), \Psi, \Pcal(\Theta))$ satisfies the conditions:
			\begin{enumerate}[(i)]
				\item $\Gamma$ is a finite union of piecewise smooth Jordan arcs, possibly unbounded but satisfying \eqref{eq:GrowthCondGamma}, and with $\Ccal \subset \Gamma$.
				\item For each $A \in \Pcal(\Ccal)$, there exists a connected component $\Gamma_A$ of $\Gamma$ that contains each point $c \in A$, and stretches to infinity in the sector $S_j$ for each $\theta_j \in \Psi(A)$.
				\item For each $B \in \Pcal_0(\Theta) \setminus \Pcal_\Psi$, there exists a connected component $\Gamma_B$ of $\Gamma$ such that $\Gamma_B$ stretches to infinity through $S_j$ for each $\theta_j \in B$.
				\item There exists $R = R(\Gamma)$ such that
				\[
					\Gamma \cap (S_j \setminus D_R) = \emptyset, \quad \text{ for each } \{\theta_j\} \in \Pcal(\Theta) \setminus \Pcal_\Psi;
				\] 
				\item $\Gamma$ has at most $|\Pcal(\Ccal)| + |\Pcal_0(\Theta) \setminus \Pcal_\Psi|$ connected components. Each one of them contains a contour $\gamma$ that satisfies at least one of the following:
				\begin{itemize}
			\item 	it contains two fixed points;
			\item it contains a fixed point and stretches to infinity through an admissible sector;
			\item it stretches to infinity through two admissible sectors.
\end{itemize}	
			\end{enumerate}
	\end{definition}
	
As a slight abuse of notation, we often write $\Tcal$ and omit the triple $\Tcal(\Pcal(\Ccal),\Psi,\Pcal(\Theta))$ from our notation, and refer to $\Tcal$ as a class of {\it admissible contours} rather than sets, to remind the reader that sets in $\Tcal$ are obtained as unions of contours.
	
	 Some expressions will be used to remove the technicality of the language of partitions given above:
\begin{itemize}
	\item $\Tcal$ \emph{connects infinity through $S_j$ and $S_k$} if $\theta_j$ and $\theta_k$ are in the same equivalence class of $\Pcal(\Theta)$. 
	\item  $\Tcal$ \emph{connects a fixed point $c$ to infinity through $S_j$} if $\theta_j \in \Psi(A)$ where $A$ is the equivalence class of $c$ in $\Pcal(\Ccal)$.
	\item  $\Tcal$ \emph{connects infinity} if any of the above is satisfied.
	\item $\Tcal$ \emph{connects $c_1$ and $c_2$} when $c_1$ and $c_2$ are in the same equivalence class of $\Pcal(\Ccal)$.
\end{itemize}

	\begin{example}\label{example:classicalchebotarev}
		The weighted Pólya-Tchebotarev's problem for a finite set $\Ccal$ can be obtained by setting $\Pcal(\Ccal) = \{\Ccal\}$, $\Psi(\Ccal) = \emptyset$ and $\Pcal_0(\Theta_\infty) = \emptyset$. 
	\end{example}
	
	\begin{example}\label{example:polchebotarev}
		Consider $\Phi(z) = V(z)$ where $V$ is a polynomial of degree $N \geq 2$, and $\Ccal = \emptyset$.  This case is exactly the one tackled in \cite{Kuijlaars2015}. Of course, class $\Tcal$ is defined only in terms of $\Pcal_0(\Theta)$.
	\end{example}

	\subsection{The main theorem}	
	
With the data given by the external field $\phi=\re\Phi$ as in \eqref{eq:PhiDef}, a collection of fixed points $\Ccal$ as in \eqref{eq:deffC}, and the class of contours $\Tcal$ for an admissible triple $(\Pcal(\Ccal),\Psi,\Pcal(\Theta))$, the associated {\it Pólya-Tchebotarev problem}, or \emph{max-min energy problem} asks for finding a set $\Gamma_0 \in \Tcal$ that maximizes the energy functional $\I^\phi$ on $\Tcal$, that is, finding $\Gamma_0$ such that
\begin{equation}\label{eq:genPC}
		\I^\phi(\Gamma_0) 
			= \max_{\Gamma \in \Tcal} \I^\phi(\Gamma)
			= \max_{\Gamma \in \Tcal} \min_{\mu \in \Mcal_1^\phi(\Gamma)} \I^\phi(\mu).
\end{equation}
	
Our main result is the following:
	\begin{theorem}\label{maintheorem}
	The max-min energy problem always has a solution $\Gamma_0 \in \Tcal$. Furthermore, $\Gamma_0$ and its equilibrium measure $\mu_0=\mu_0(\phi,\Gamma_0)$ satisfy the following properties
			\begin{enumerate}[(1)]
			\item $\Gamma_0$ is a finite union of analytic arcs.
			\item The function
			$$
			R(z) \defeq \left(  \CT^{\mu_0}(z) + \Phi'(z)\right)^2, \quad z \in \C \setminus \supp \mu_0,
			$$
			extends to a rational function on $\C$. Its simple poles belong to the set of fixed points $\Ccal$, its double poles coincide with points $w_1, \ldots, w_{n_L}$, and all its other poles have even order, always at least $4$, and coincide with the poles of the rational part of $\Phi$.
			\item The equilibrium measure $\mu_0$ is supported on a finite union of analytic arcs that are critical trajectories of the quadratic differential $\varpi=-R(z) \dd z^2$, it is absolutely continuous in each analytic subarc of $\Gamma_0$ with respect to the arc-length measure, with density
				$$
					\dd \mu_0(s) = \frac{1}{\pi \ii} R_+(s)^{1/2} \, \dd s.
				$$
			Furthermore, the simple poles of $R$ are necessarily in $\supp\mu_0$, and the remaining poles of $R$ do not belong to $\supp\mu_0$. 
			\end{enumerate}
	\end{theorem}

As a consequence of the discussion in Examples~\ref{example:classicalchebotarev} and \ref{example:polchebotarev}, Theorem~\ref{maintheorem} includes both the Classical Pólya-Tchebotarev Problem and the max-min considered in \cite{Kuijlaars2015} as particular cases.

As mentioned earlier, max-min problems are intimately connected with the zeros distribution of non-hermitian orthogonal polynomials. In this regard, a standard calculation shows that the measure $\mu_0$ satisfies the property
$$
\frac{\partial}{\partial \eta_+}\left(\U^{\mu_0}(z)+ \phi(z)\right)=\frac{\partial}{\partial \eta_-}\left(\U^{\mu_0}(z)+\phi(z)\right),
$$
on subarcs of $\supp\mu_0$, where $\eta_\pm$ are the two unit vectors normal to $\supp\mu_0$ at the point $z\in \supp\mu_0$. This property is known as the {\it S-property} in orthogonal polynomial theory.

	\subsection{Strategy}
	
			In \cite{Rakhmanov2012} Rakhmanov presented an ``ideal situation'' under which a max-min weighted log energy problem admits a solution. However, such ideal situation is not always in place in many important concrete situations. In fact, an example of the type of problems for which his approach does not apply directly was mentioned already in \cite{Rakhmanov2012}, in the context of non-hermitian orthogonality with polynomial external fields. His approach was then adapted to this situation by Kuijlaars and the second-named author in \cite{Kuijlaars2015}, and our approach here follows their footsteps. When compared with \cite{Kuijlaars2015}, many of the technical lemmas need adaptation due to the fact that we now have to deal with singularities of $\Phi$ at finite points, not only in $\infty$. We highlight the distinctions along the way of our proofs.
		
		To explain Rakhmanov's approach, define the hyperbolic metric on $\overline{\C}$ as the pull-back of the Euclidean metric on $T(\overline{\C}) \subset \R^3$, where $T$ is the homeomorphism
		$$
			T(z) \defeq \left( \frac{\Real z}{1 + |z|^2}, \frac{\Img z}{1 + |z|^2}, \frac{|z|^2}{1 + |z|^2} \right), \quad z \in \C,
		$$
		with $T(\infty) = (0,0,1)$. The image of $T$ is the sphere centered at $(0,0,\frac{1}{2})$ with radius $\frac{1}{2}$. For points $z,w \in \C$ we have the following explicit formula in terms of the usual Euclidean norm of $\C$,
		$$
			d_H(z,w) = \frac{|z-w|}{\sqrt{1 + |z|^2} \sqrt{1 + |w|^2}}.
		$$
		In addition, for $K\subset \overline{\C}$ set
		$$
		(K)_\delta = \left\{ x \in \overline{\C}; \; d_H(x,K)=\sup_{y\in K}d_H(x,y) < \delta \right\}.
		$$
		The Hausdorff metric $d_H(K_1,K_2)$ on compacts $K_1,K_2 \subset \overline{\C}$ induced by $d_H$ is
		$$
			d_H(K_1,K_2)  
				= \inf \left\{ \delta > 0;\; K_1 \subset (K_2)_\delta \text{ and } K_2 \subset (K_1)_\delta \right\}.
		$$
		The Hausdorff metric above naturally induces a metric on the closed sets of $\C$, given by the distance between their closures on $\overline{\C}$. 
		
		Now let $\Fcal$ be a family of closed sets on $\C$. Assume in addition that the number of connected components of the elements of $\Fcal$ is uniformly bounded by a finite value. Equip $\Fcal$ with the Hausdorff metric just mentioned. 
		\begin{theorem}[\protect{\cite[Theorem 3.2]{Rakhmanov2012}}]\label{RakhmanovTheorem}
		Suppose that $\phi=\re\Phi$ is harmonic in $\overline\C$, except at finitely many points. Then, the energy functional $\I^\phi: \Fcal \to [-\infty, +\infty]$ is upper semicontinuous.
		\end{theorem}
		
		If the class $\Fcal$ is closed and $\I^\phi$ is bounded from above by a finite value, then Theorem~\ref{RakhmanovTheorem} implies the existence of a set $F_0$ maximizing $\I^\phi$ on $\Fcal$. This is not true for an admissible class of contours $\Tcal$ as in Definition~\ref{def:AdmissibleClass}. Indeed, first of all, contours are not preserved under Hausdorff metric: one could construct a sequence of space-filling curves converging to a set of strictly positive measure on $\C$, hence not a contour. But, in a more drastic issue, one could have a sequence of contours where each one extends to $\infty$ along given fixed sectors, but with a limiting contour that does not extend on the given sector, but in a different (neighboring) one. 
		
		To overcome these issues, we need instead to work with a subclass $\Tcal_0$ of $\Tcal$, suitably chosen so that sets in the closure $\Fcal_0=\overline{\Tcal}_0$ of $\Tcal_0$ in the Hausdorff metric have a better geometry than sets in the closure $\Fcal=\overline{\Tcal}$ of $\Tcal$, in order to avoid the second issue mentioned above. 
		
Nevertheless, $\Fcal_0$ still consists of sets more arbitrary than unions of contours. Thus, when we apply Theorem~\ref{RakhmanovTheorem} to $\Fcal_0$, we obtain a max-min set $F_0$ for the energy $\I^\phi$ which is not necessarily a contour. 

The second step is to study the equilibrium measure $\mu_0=\mu_*(\phi,F_0)$ of $F_0$, which turns out to be a {\it critical measure}, that is, a saddle point of $\I^\phi$ when viewed as acting on appropriate measures on $\C$. Using this interpretation, we prove several qualitative features for $\mu_0$. In particular, using the basic theory of quadratic differentials we show that $\supp\mu_0$ consists of a union of arcs.

With $\mu_0$ at hand, we then reconstruct a contour $\Gamma_0\in \Tcal_0$ for which $\mu_*(\phi,\Gamma_0)=\mu_0$. In this process, the fact that the original set $F_0$ belongs to the closure of $\Tcal_0$ is essential, with the properties that define $\Tcal_0$ being the crucial ingredient.
	
\subsection{Structure for the rest of the paper}
	
The organization of the rest of the paper is as follows.

In Section \ref{sec:Critical}, we present a systematic study of the \emph{critical measures} associated to the max-min problem, which are the critical points of the energy functional $\I^\phi$. We obtain structural results on them, showing in particular that they are supported on curves, and relating them to the poles of $\Phi$. The obtained geometric properties of their support play a substantial role in the later construction of the max-min solution.

After having a suitable characterization for critical measures we study their analogue for sets, the so-called {\it critical sets}. This is done in Section \ref{sec:CritSets}, where we also establish that the equilibrium measure of a critical set with finite energy is a critical measure, and thus translate results from critical measures to critical sets.

At last, we turn focus to the Pólya-Tchebotarev Problem \eqref{eq:genPC}, and in Section \ref{sec:MaxMinSolution} we combine the results of previous sections to prove Theorem~\ref{maintheorem}.  

In Appendix~\ref{sec:QDs} we review some results from the general theory of quadratic differentials that we use during the text, in a way suitably adapted to our purposes.

\section{Critical Measures and Critical Sets}\label{sec:Critical}

As explained previously, the critical points of the energy functional $\I^\phi$, both viewed as acting on measures and on sets, play a fundamental role in our approach. This section is dedicated to their study.

	\subsection{Critical Measures}\label{sec:criticalmeasures}
	
	While studying the asymptotic behavior of zeroes of Stieltjes Polynomials, Martínez-Finkelshtein and Rakhmanov defined the notion of  \emph{critical measures}, measures that are critical points of a log energy functional. Since then this notion has found its way to other problems, for instance, in zero distribution of orthogonal polynomials and related families \cite{MartinezFinkelshteinSilva2016, Bertola2022a}, random matrix theory and other related models \cite{MartinezFinkelshteinSilva2021}.

We always assume $\phi=\re\Phi$, where $\Phi$ was defined in Equation \eqref{eq:PhiDef}. We also fix the notation  $\Zcal=\{z_1,\hdots, z_{n_Q},w_1,\ldots,w_{n_L}\}$ for the set of singularities, $\Wcal = \{ w_1,\ldots,w_{n_L} \} \setminus \Zcal$  and $\Ccal=\{c_1,\hdots, c_{n_\Ccal}\}$ for the set of fixed points. We also consider the polynomial
\[
	C(z) = \prod_{i=1}^{n_C} (z - c_i)
\]
	with $C \equiv 1$ when $\Ccal = \emptyset$.

 For this section, it is also convenient to write
\begin{equation}\label{eq:partialfractionPhiprime}
\Phi'(z)=\frac{B(z)}{A(z)}, 
\quad \text{with}\quad A(z)\defeq \prod_{j=1}^{n_Q}(z-z_j)^{m_j+1} \prod_{w \in \Wcal} (z-w_j) ,
\end{equation}
 where $B$ is a polynomial of degree $\deg P + n_Q + | \Wcal| - 1$ with $B(z)\neq 0$ for $z\in \Wcal$. 

We consider the following class of test functions
\begin{equation}\label{deff:testfunctions}
\Hcal \defeq \{ h\in C_c^2(\C); \; h(c)=0, c\in \Ccal, \text{ and }  h\equiv 0 \text{ in a neighborhood of } \Zcal \}.
\end{equation}
In other words, $\Hcal$ consists of the class of compactly supported, complex-valued functions $h:\C\to \C$, which are twice (real) differentiable, that vanish at the points in $\Ccal$, and that vanish in a neighborhood of each singularity of $\Phi$. 
	
\begin{definition} \label{deff:variationenergymeasure}
Let $\mu$ be a finite measure on $\C$ and $h\in \Hcal$. The \emph{derivative of the weighted energy functional $\I^\phi$ at $\mu$ in the direction of $h$} is the limit
	$$
		\D_h \en^\phi(\mu) \defeq \lim_{t \to 0} \frac{\I^\phi(\mu^t) - \I^\phi(\mu)}{t},
	$$
	if it exists, where $\mu^t$ is the pushforward of $\mu$ by $ z \mapsto z^t \eqdef z + th(z)$.
	\end{definition}

The properties of $\Hcal$ are tied to our later goal of perturbing a max-min contour, and using its extremal property within $\Tcal$ to derive properties of its equilibrium measure, which will turn out to be critical. We impose $h$ to vanish at $\Ccal$ so as that perturbations of the extremal contour indeed remain in $\Tcal$. Later on, we will in fact show that critical measures cannot be supported near singularities of $\Phi$, but until we are able to prove this fact, we need to impose that members of the class $\Hcal$ vanish near the singularities of $\Phi$, to avoid a blow-up of the energy when doing perturbations.
	
Recall that $\Mcal_1^\phi(\C)$ was previously defined in Section~\ref{subsec:PotTheory}. It turns out that, within the class of test functions $\Hcal$, directional derivatives of $\mu\in \Mcal_1^\phi(\C)$ always exist and are somewhat explicit, as claimed by the next proposition. This proposition has appeared before, for instance, in \cite{MartinezFinkelshtein2011, Kuijlaars2015}, but for completeness the proof here has to be adapted in virtue of the singularities of $\Phi$.
	
	\begin{proposition}\label{MuDerivative}
		For each $\mu \in \Mcal_1^\phi(\C)$ and $h \in \Hcal$, the derivative $\D_h \I^\phi(\mu)$ exists and is given by
		\begin{equation}
			\D_h\I^\phi(\mu) 
			= - \Real \left( \iint \frac{h(x)-h(y)}{x-y}  \, \dd \mu(x) \, \dd \mu(y) - 2 \int \Phi'(x)h(x) \, \dd \mu(x) \right).
		\end{equation}
	\end{proposition}
	
	\begin{proof}
		By definition of pushforward measures,
		$$
			\I^\phi(\mu^t) 
			= - \iint \log |x-y+t(h(x)-h(y))| \, d\mu(x) \, \dd \mu(y) 
				+ 2\int \phi(x+th(x)) \, \dd \mu(x),
		$$
		therefore
		\begin{equation}\label{energydifference}
			\I^{\phi}(\mu^t) - \I^\phi(\mu) 
				= - \iint \log \left| 1 + t \frac{h(x)-h(y)}{x-y} \right| \, \dd \mu(x) \, \dd \mu(y) 
				+ 2 \int [ \phi(x+th(x)) - \phi(x) ] \, \dd \mu(x).
		\end{equation}
		
		For small $t$ we have
		\begin{equation}\label{logbh}
		\begin{split}
			\log \left|1 + t \frac{h(x)-h(y)}{x-y} \right|
				&= \Real \log \left( 1 + t \frac{h(x)-h(y)}{x-y} \right) \\
				&= \Real \left( t \frac{h(x)-h(y)}{x-y} \right) +  o(t) 
		\end{split}
		\end{equation}
			where, since $h \in C^2_c$, the constant in $o$ is uniform with respect to $x,y \in \C$.

		For the second integral start by writing
		$$
			\int \left[ \phi(x+th(x)) - \phi(x) \right] \, \dd \mu(x) = \Real  \int_{\C \setminus O} \left(\Phi(x+th(x)) - \Phi(x)\right) \, \dd \mu(x) ,
		$$
		where $O$ is a finite union of open discs $D_{\eta/2}(z_j)$ in which $h$ is identically zero. 
		
		If $x$ and $x+th(x)$ are not singularities of $\Phi$, we write
		$$
			\Phi(x+th(x)) - \Phi(x) = \Phi'(x) th(x) + r(th(x);x),
		$$
		where the term $r$ should be interpreted as an error.
		In particular, from this identity we obtain that $r(th(x);x) = 0$ whenever $x \notin \supp h$.
		
		For sufficiently small $t$ one has that $x \in \C \setminus O$ implies $x + th(x) \in \C \setminus O$. Thus
		$$
			|r(th(x);x)| \leq \frac{\sup_{c \in [x,x+th(x)]}|\Phi''(c)|}{2} t^2 |h(x)|^2.
		$$
		Since $x \in \supp h$ and $c \notin D_\eta(z_j)$ for every $x$, there exists a  uniform constant $M > 0$ such that
		$$
			|r(th(x);x)| \leq M t^2 \|h^2\|_\infty.
		$$
		 This implies
		\begin{equation}\label{phibh}
			\phi(x+th(x)) - \phi(x) = t \Real(\Phi'(x)h(x)) + o(t), \quad t \to 0.
		\end{equation}
		Plugging (\ref{logbh}) and (\ref{phibh}) in (\ref{energydifference}) we get
		$$
			\I^\phi(\mu_t) - \I^\phi(\mu) = - t \Real \left( \iint \frac{h(z)-h(x)}{z-x} \, \dd\mu(z) \, \dd\mu(x) - 2\int \Phi'(x)h(x) \, \dd\mu(x) \right) + o(t)
		$$
		and the result follows.
	\end{proof}
	
Now that we have a notion of derivative of $\I^\phi$, the notion of critical measure is natural.
	
	\begin{definition}\label{def:criticalmeasure}
		A measure $\mu \in \Mcal^\phi_1(\C)$ is  called \emph{critical} when 
		$$
			\D_h \I^\phi(\mu) = 0
		$$
		for every $h \in \Hcal$.
	\end{definition}
	
	Considering $h$ and $\ii h$ separately, the next result is an immediate consequence of Proposition~\ref{MuDerivative}.
	
	\begin{corollary}\label{Cor:critical}
		A measure $\mu \in \Mcal^\phi_1(\C)$ is critical if, and only if, 
		\begin{equation}\label{criteq}
			\iint \frac{h(x)-h(y)}{x-y} \, \dd\mu(x) \, \dd\mu(y) = 2 \int \Phi'(x) h(x) \, \dd\mu(x)
		\end{equation}
		for every $h \in \Hcal$.
	\end{corollary}

	Martínez-Finkelshtein and Rakhmanov \cite{MartinezFinkelshtein2011} showed that the Cauchy transform of a critical measure satisfies an algebraic equation, which in turn guarantees that the support of the measure lives in trajectories of a quadratic differential. Analogous results were obtained in \cite{Kuijlaars2015, MartinezFinkelshteinSilva2016}. The main result of this section establishes the same connection between critical measures and algebraic equations. Following \cite{MartinezFinkelshtein2011, Kuijlaars2015}, the main idea is to apply \eqref{criteq} to the so-called Schiffer variations $x\mapsto h(x)=H(x)/(x-\zeta)$, for each $\zeta\in \C$ fixed and $H(x)$ an appropriately chosen polynomial. However, to do so we need to extend the validity of Corollary~\ref{Cor:critical} beyond the class of test functions $\Hcal$, which we do so extending it in several steps. But before proceeding, we first establish an auxiliary approximation lemma.
	
	\begin{lemma}\label{lem:approx}
		For every $n$ positive integer, there exists $C^2$ functions $\alpha_n, \beta_n :(0,+\infty) \to [0,1]$ such that $\alpha_n, \beta_n \to 1$ pointwise and
		$$
			\beta_n |_{(0,n)} , \alpha_n |_{(2/n,+\infty)} \equiv 1, 
			\quad \beta_n|_{(3n,+\infty)} , \alpha_n|_{(0,1/n)} \equiv 0,
			\quad  |\beta_n'(t)| \leq \frac{1}{n}, 
			\quad |\alpha_n'(t)|  \leq K t^{-1}.
		$$
		for some $K > 0$ independent of $n$.
	\end{lemma}
	
	\begin{proof}
		A cumbersome but straightforward calculation shows that the functions
		$$
		\beta_n(t) \defeq
				\begin{cases}
					1, \quad t \leq n \\
					\frac{1}{16n^5} (3n-t)^3 (2n^2 -3nt + 3t^2), \quad n < t < 3n , \\
					0, \quad 3n \leq t,
				\end{cases}
		$$
		and
		$$
			\alpha_n(t) \defeq
			\begin{cases}
			0, \quad t \leq \frac{1}{n},\\
			 (nt-1)^3 (6n^2t^2 - 27nt + 31), \quad \frac{1}{n} \leq t \leq \frac{2}{n} \\
			 1, \quad \frac{2}{n} \leq t
			 	\end{cases}
		$$
		satisfy all conditions above. 
	\end{proof}

For the next result, recall that $C$ is the polynomial defined in \eqref{eq:deffC}, and the polynomials $A$ and $B$ were introduced in \eqref{eq:partialfractionPhiprime}.
	
	\begin{lemma}\label{crit1}
	Fix $M \defeq \deg B + \deg C$ points $x_1,\ldots,x_M$ which are not singularities of $\Phi$, and a function $\chi\in C^2(\C)$ which vanishes in a neighborhood of each point $x_1,\hdots,x_M$, and is identically equal to $1$ in a neighborhood of $\infty$.
If $\mu \in \Mcal_1^\phi(\C)$ is a critical measure, then	Equation \eqref{criteq} remains valid for the choice
		 \begin{equation}\label{eq:CritiEqGen1}
		 	h(z) \defeq \frac{A(z) C(z)}{\prod_{j=1}^M (z-x_j)} \chi(z),\quad z\in \C.
		 \end{equation}
	\end{lemma}
	
	\begin{proof}
	We shall approximate $h$ by a sequence of functions $(h_n)$ in $\Hcal$, such that the integrals in \eqref{criteq} with respect to $h_n$ converge to the corresponding integrals with respect to $h$.

With $(\alpha_n),(\beta_n)$ as in Lemma~\ref{lem:approx}, define
		\begin{equation}\label{eq:hnhapproxlemma}
			\psi_n(z) \defeq \beta_n(|z|) \prod_{w \in \Zcal} \alpha_n(|z-w|), \quad 
			h_n(z) \defeq    \psi_n(z) h(z)   , \quad z \in \C.
		\end{equation}

		Note that since $\alpha_n' \equiv 0$ in a neighborhood of $0$,
		\begin{equation*}
		\begin{split}
			\frac{\partial}{\partial \overline{z}} \alpha_n(|z-w|) &= \alpha_n'(|z-w|) \frac{1}{2} \frac{z-w}{|z-w|} \in C^1, \\
			\frac{\partial}{\partial z} \alpha_n(|z-w|) &= \alpha_n'(|z-w|) \frac{1}{2} \frac{\overline{z-w}}{|z-w|} \in C^1.
		\end{split}
		\end{equation*}
		Thus $\alpha_n \in C^2$. Analogously one sees that $\beta_n \in C^2_c$. This implies that $h_n \in \Hcal$, so \eqref{criteq} is valid for it.

		From \eqref{eq:CritiEqGen1} and $M = \deg B + \deg C$, we see that $\Phi'h$ is continuous and bounded, and thus belongs to $L^1(\dd\mu)$. Since $|\alpha_n|,|\beta_n|\leq 1$, we have $|\Phi'h_n|\leq |\Phi' h|$, and the Dominated Convergence Theorem implies that
		\begin{equation}
			\lim_{n \to \infty} \int \Phi'(x) h_n(x) \, \dd\mu(x) = \int \Phi'(x) h(x) \, \dd\mu(x).
		\end{equation}
		
		Because $\mu\in \Mcal_1^\phi(\C)$, the measure $\mu$ has finite logarithmic energy, so it has no mass points. In particular, the diagonal $\{(x,x); x \in \C\}$ has zero $\mu \times \mu$ measure, and therefore
		$$
			\lim_{n \to \infty} \frac{h_n(x) - h_n(y)}{x-y} = \frac{h(x) - h(y)}{x-y}, \quad \mu \times \mu- \text{a.e.}.
		$$
		
		The Mean Value Inequality implies that
		$$
			\left| \frac{h_n(x) - h_n(y)}{x-y} \right| 
			\leq \sup_{z \in \C} \| D h_n(z)\|.
		$$
	Thus, provided that we show that the right-hand side is finite and bounded uniformly in $n \in \N$, the Dominated Convergence Theorem gives
		$$
			\lim_{n \to \infty} \iint \frac{h_n(x) - h_n(y)}{x-y} \, \dd\mu(x) \, \dd\mu(y) = \iint \frac{h(x) - h(y)}{x-y} \, \dd\mu(x) \, \dd\mu(y).
		$$
		
		We now verify that $\| Dh_n(z)\|$ is bounded uniformly in $n$.
From a direct calculation,
		\begin{multline*}
			2 \frac{\partial}{\partial \overline{z}} \psi_n(z) 
				 = \frac{z}{|z|}\beta_n'(|z|)    \prod_{w \in \Zcal} \alpha_n(|z-w|) \\ 
					+ \beta_n(|z|)   \sum_{w \in \Zcal} \left( \frac{z - w}{|z- w|} \alpha_n'(|z-w|) \left[ \prod_{ \substack{w' \in \Zcal \\ w' \neq w}} \alpha_n(|z-w'|)\right] \right),
		\end{multline*}
and therefore using Lemma~\ref{lem:approx},
\begin{equation}\label{eq:boundgradientpsin}
			 \| \nabla \psi_n(z) \| 
			= 2 \left|\frac{\partial }{\partial \overline{z}} \psi_n(z)\right| \leq 1  + K \sum_{w \in \Zcal} |z-w|^{-1}
\end{equation}
We view $h=h_1+\ii h_2\simeq (h_1,h_2)$ as a function $\R^2\to \R^2$, with $h_1=\Re h, h_2=\Im h$ and Jacobian $Dh$. For $z \in \C$  and $(u,v) \in \R^2$, we differentiate \eqref{eq:hnhapproxlemma} and obtain
		\begin{equation*}
		\begin{split}
			D h_n(z) \cdot (u,v) 
				&= \left[ \psi_n(z) D h(z) + 
				\begin{pmatrix}
				h_1(z) \nabla \psi_n(z) \\ h_2(z) \nabla \psi_n(z)
				\end{pmatrix}  \right] \cdot (u,v).
				\end{split} 
		\end{equation*}
Thus, taking operator norm
		$$
			\| D h_n(z)\| \leq \|D h(z)\| + \left\| \begin{pmatrix}
				h_1(z) \nabla \psi_n(z) \\ h_2(z) \nabla \psi_n(z) 
				\end{pmatrix}\right\|.
		$$
		Since $h_1(z) = \Real h(z)$, $h_2(z) = \Img h(z)$ and $h$ has a zero of order at least $1$ at each $w \in \Zcal$, we use \eqref{eq:boundgradientpsin} and conclude that the second term in the r.h.s. above is bounded.
		
		For the first term, note that $\|D h(z)\|$ is bounded on compact sets since $h \in C^2$. Using that $h$ is holomorphic for sufficiently large $z$, we see that
		$$
			Dh(z) \cdot (u,v) = h'(z) (u+iv).
		$$
		Therefore $\|Dh(z)\| = |h'(z)|$ is bounded a neighborhood of infinity and, thus, on the whole complex plane.	
		This concludes the proof.
	\end{proof}	
	
Recall that $\Zcal$ denotes the set of singularities of $\Phi$, see \eqref{eq:defZ}.	
	
	\begin{lemma}\label{crit2}
		Let $\mu\in \Mcal_1^\phi(\C)$ be a critical measure. If $z \in \C \setminus \Zcal$ satisfies
		\begin{equation}\label{zcondition}
			\int \frac{1}{|x-z|}\, \dd\mu(x) < \infty,
		\end{equation} 
then Equation (\ref{criteq}) remains valid for 
		$$
			h(x) = \frac{g(x)}{x-z}, \quad g \in \Hcal.
		$$
		\end{lemma}
	
\begin{proof}
When $\Zcal=\emptyset$, that is, when $\Phi$ is a polynomial, the result reduces to \cite[Lemma~3.5]{Kuijlaars2015}. The proof presented here is essentially the same, we briefly outline it and refer to \cite{Kuijlaars2015} for more details.

	The idea behind the proof is the same as the one on Lemma \ref{crit1}, we shall approximate $h$ by a sequence of functions $h_n \in \Hcal$ whose integrals in $\eqref{criteq}$ with respect to $h_n$ converge to the corresponding integrals with respect to $h$.
	
	Write
	\[
		h(x) = \frac{g(x) \overline{(x-z)}}{|x-z|^2}
	\]
	and define
	\[
		h_n(x) =\frac{g(x) \overline{(x-z)}}{|x-z|^2 + \epsilon_n^2}
	\]
	where $\epsilon_n$ is a sequence of positive numbers converging to zero. Note that $h_n \in \Hcal$ and thus \eqref{criteq} holds for it. Now, Inequality \eqref{zcondition} implies that $h \Phi'$ is integrable with respect to $\mu$ and the dominated convergence theorem implies
	\[
		\int h_n(x) \Phi'(x) \, \dd \mu(x) \to \int h(x) \Phi'(x) \, \dd \mu(x).
	\]
	
	One can show that
	\[
		\left| \frac{h_n(x) - h_n(y)}{x-y} \right| \leq \left( \frac{1}{|x-z|} + \frac{1}{|y-z|} \right) \left| \frac{g(x) - g(y)}{x-y}\right| + \frac{2}{|y-z|} \left|\frac{g(x)}{x-z}\right|
	\]
	and since $g \in \Hcal$, there exists $M > 0$ such that 
	\[
		\left| \frac{h_n(x) - h_n(y)}{x-y} \right| \leq  \frac{M}{|x-z|} + \frac{M}{|y-z|}  + \frac{M}{|x-z||y-z|}.
	\]
	Now, Inequality \eqref{zcondition} implies that the right-hand side above is $\mu \times \mu$ integrable, and thus the Dominated convergence theorem implies
	\[
		\iint \frac{h_n(x) - h_n(y)}{x-y} \, \dd \mu(x) \, \dd \mu(y) \to \iint \frac{h(x) - h(y)}{x-y} \, \dd \mu(x) \, \dd \mu(y).
	\]
	
\end{proof}

With a combination of the previous results, we now obtain that the function $\chi$ may be removed in Lemma~\ref{crit1} when the chosen points $x_1,\ldots,x_M$ satisfy Equation \eqref{zcondition}.
			
	\begin{corollary}\label{crit3}
		Let $\mu\in \Mcal_1^\phi(\C)$ be a critical measure. 
		Fix a set $X = \{x_1,\ldots,x_{M} \}$ of $M \defeq \deg B + \deg C$ distinct points, none of them singularities of $\Phi$, satisfying
		$$
			\int \frac{1}{|x - x_j|} \, \dd\mu(x) < +\infty, \quad j = 1 , \ldots,  M.
		$$
		Then, Equation \eqref{criteq} holds for 
		\begin{equation}\label{criteq2}
			h(x) = \frac{A(x) C(x)}{\prod_{j=1}^M(x-x_j)}.
		\end{equation}
	\end{corollary}
	
	\begin{proof}
		For this proof, set
		$$
			\delta \defeq \frac{1}{4} \min \left\{ |x - y|, x,y \in X \cup \Zcal, x \neq y  \right\}.
		$$
		For each $1 \leq j \leq M$, fix a $C_c^2$ function $\chi_j: \C \to [0,1]$, supported on $D_{2\delta}(x_j)$ with $\chi_j \equiv 1$ on $D_\delta(x_j)$. Introduce
		$$
			g_j(x) \defeq \left(\prod_{\substack{i=1\\ i \neq j}}^M \frac{1}{x - x_i} \right)  A(x) C(x) \chi_j(x)  \quad \text{and}\quad 	
			h_j(x) \defeq \chi_j(x) h(x) = \frac{g_j(x)}{x - x_j}.
		$$
		Then $g_j\in \Hcal$, so $h_j$ satisfies the condition of Lemma \ref{crit2}. On the other hand, 
		$$
			h(z) - \sum_{j=1}^M h_j(z) = \frac{A(z)C(z)}{\prod_{j=1}^M(x-x_j)}\chi(z), \quad \chi(z) \defeq 1 - \sum_{j=1}^M \chi_j(z)
		$$ 
		satisfies the hypotheses of Lemma \ref{crit1}. Thus \eqref{criteq} is valid for each $h_j$ and also for $h-\sum h_j$. Consequently, by linearity, \eqref{criteq} also holds for $h$.
	\end{proof}
	
	\begin{remark}
	The introduction of the points $x_1,\hdots, x_M$ was done for technical reasons. In essence, it is present to ensure that with the choice \eqref{criteq2}, the product $h\Phi'$ is integrable near $\infty$. Later on, we will show that critical measures are in fact compactly supported, which then will mean that we can replace the factor $\prod (x-x_j)$ simply by $1$ in all the proofs.
	\end{remark}
		
For the next result, recall that $\CT^\mu$ denotes the Cauchy transform of the measure $\mu$, see \eqref{deff:logpotCauchytransf}. The next result is one of the key steps in establishing Theorem~\ref{maintheorem}. Its main consequence is that the Cauchy transform of a critical measure satisfies an algebraic equation, namely Equation~\eqref{QuadEq}. The latter is known in the context of random matrix theory as the spectral curve or master loop equation.

Different versions of this result have been obtained in \cite{MartinezFinkelshtein2011} and \cite{Kuijlaars2015}. Here we adapt the proof of both mentioned papers, to accommodate for the singularities of $\Phi$ and the set of fixed points $\Ccal$.

	\begin{proposition}\label{prop:QuadraticEquation}
		Let $\mu\in \Mcal_1^\phi(\C)$ be a critical measure. Then, there exists a rational function $R$ such that
		\begin{equation}\label{QuadEq}
			\left( \CT^\mu(z) +  \Phi'(z)  \right)^2 = R(z), \quad \Leb-\text{a.e.}
		\end{equation}
		where $\Leb$ is the Lebesgue measure on $\C$. Furthermore, $R$ has the following properties.
		\begin{enumerate}[(i)]
		\item Each $z_j$ of order $m_j$ of $\Phi$ is a pole of order $2m_j+2$ of $R$.
		\item Each $w_i \in \Wcal$ is a double pole of $R$.
		\item Simple poles of $R$ are contained in the set of fixed points $\Ccal$.		
		\item There are no poles other than the ones in (i)--(iii).
		\item $R(z)=b^2 z^{2\deg B-2\deg A}(1+\Ord(z^{-1})), z\to \infty$, where $b\neq 0$ is the leading coefficient of the polynomial $B$.
		\end{enumerate}
	\end{proposition}
	\begin{proof}
		Recall that we denote $M =\deg B + \deg C$. The function
		$$
		z\mapsto \int \frac{\dd\mu(x)}{|x-z|}
		$$
		may be viewed as the convolution of $1/|x|\in L^1_{\mathrm{loc}}(\dd \Leb)$ with the measure $\mu$. Hence, this function belongs to $L^1(\dd \Leb)$, so it is finite $\Leb$-a.e.. In particular, we can choose points $x_1,\hdots, x_{M-1} \notin \Zcal$ such that
		\begin{equation}\label{condzj}
			\int \frac{1}{|x-x_j|} \, \dd \mu(x) < +\infty, \quad 1 \leq j \leq M-1,
		\end{equation}
		and furthermore it suffices to prove Equation \eqref{QuadEq} for $z\in  \C \setminus (\{x_1,\ldots,x_{M-1}\}\cup \Zcal )$ for which 
		\begin{equation}\label{condz2}
		\int \frac{1}{|x-z|} \, \dd \mu(x) < +\infty.
		\end{equation}

		Define
		$$
			p(x) \defeq \prod_{j=1}^{M-1} \frac{1}{(x - x_j)}, \quad
			h(x) = h(x;z) \defeq \frac{p(x)}{x-z} A(x) C(x).
		$$
The function $h$ satisfies the conditions of Corollary \ref{crit3}, and therefore Equation \eqref{criteq} is valid for it. Our next goal is to rewrite both sides of \eqref{criteq} for this $h$.
		
We start by writing
		\begin{equation}\label{D_1}
			\int \Phi'(x) h(x) \, \dd \mu(x) = p(z) B(z) C(z) \CT^\mu(z) + D_1(z),
		\end{equation}
		where  
		$$
			D_1(z) \defeq \int \frac{p(x) B(x) C(x) -  p(z) B(z) C(z)}{x-z} \, \dd\mu(x).
		$$
		We claim that the only possible poles of $D_1$ are the points $x_1,\ldots,x_{M-1}$. To see that, start by putting $q(w) = 1/p(w)$ and write 
		\begin{align*}
			\frac{p(x) B(x) C(x) -  p(z) B(z) C(z)}{x-z} 
			&= p(x)p(z) \left[B(x)C(x)q(z) - B(z)C(z)q(x) \right] \frac{1}{x-z}
		\end{align*}
		The expression in brackets is a polynomial in $x$ and $z$. Fixing $z$, we see it as a polynomial in $x$ that has a root in $x=z$. Thus, it can be written as 
		$$
			(x-z) P_1(x,z)
		$$
		where $P_1$ is a polynomial in $x$ and $z$ of degree $M - 1$, so
		$$
			\frac{p(x) B(x) C(x) - p(z) B(z)C(z)}{x-z} = p(x)p(z) P_1(x,z).
		$$
		This yields the representation
		\begin{equation}\label{eq:reprD1int}
			D_1(z) = p(z) \int p(x) P_1(x,z) \, \dd \mu(x).
		\end{equation}
		Note that the integral above converges. In fact, it is sufficient to show that it converges in a neighborhood of $x = x_j$ and $x = \infty$. As  $x \to x_j$, $P_1(x,z)$ is bounded and $(\ref{condzj})$ implies that it converges in a neighborhood of $x_j$. Also, since $p(x) P_1(x,z) = \Ord(1)$ as $x \to \infty$, the integral too converges near $x=\infty$. Thus, the integral on the right-hand side of \eqref{eq:reprD1int} yields a polynomial of degree $M-1$ in $z$, hence $D_1$ has possible poles only at the poles of $p$ itself, and which are necessarily simple. Furthermore we also conclude that
\begin{equation}\label{eq:behD1}
		D_1(z)=\Ord(1)\quad \text{as }z\to\infty.
\end{equation}

		Next, we shift our attention to the quotient $\frac{h(x) - h(y)}{x-y}$. Start by writing
		\begin{multline*}
			\frac{h(x) - h(y)}{x-y} 
			= \\ 
			\frac{(y-z)p(x)A(x)C(x) 
			+ (z-x)p(y)A(y)C(y) 
			+ (x-y)p(z)A(z)C(z)}{(x-y)(x-z)(y-z)} - \frac{p(z)A(z)C(z)}{(x-z)(y-z)}.
		\end{multline*}
		Note that the second term on the right hand side is $\dd \mu(x) \times \dd \mu(y)$ integrable since we are assuming (\ref{condz2}). It follows that
\begin{equation}\label{eq:D2}
			\iint \frac{h(x)-h(y)}{x-y} \, d\mu(x) \, d\mu(y) = -p(z)A(z)C(z)(C^\mu(z))^2 + D_2(z),
\end{equation}
		where
		$$
			D_2(z) \defeq \iint \frac{(y-z)p(x)A(x)C(x) + (z-x)p(y)A(y)C(y) + (x-y)p(z)A(z)C(z)}{(x-y)(x-z)(y-z)} \, d\mu(x) \, d\mu(y).
		$$
An argument similar to the one we just performed for $D_1$ shows that $D_2$ may be represented in the (convergent) double integral formula
$$
D_2(z)=p(z)\iint p(x)p(y)P_2(x,y,z)\dd\mu(x)\dd\mu(y),
$$		
where $P_2$ is a polynomial of degree at most $M-2$ in each of its variables $x,y,z$. This shows that $D_2$ is a rational function as well, with only simple poles, all contained in the set $\{x_1,\hdots, x_{M-1}\}$, and furthermore
\begin{equation}\label{eq:behD2}
			D_2(z) = \Ord(z^{-1}), \quad z \to \infty.
\end{equation}
Combining \eqref{D_1} and \eqref{eq:D2} into \eqref{criteq}, we obtain the identity
		$$
			D_2(z) - p(z)A(z)C(z)(C^\mu(z))^2 = 2p(z)\Phi'(z)A(z)C(z)C^\mu(z) + 2 D_1(z).
		$$
		Now, this is equivalent to (\ref{QuadEq}) if we put
\begin{equation}\label{eq:aux1}
		\begin{aligned}
			R(z) &= \left(\frac{D_2(z) - 2 D_1(z)}{p(z)} \right) \frac{1}{A(z)C(z)} +  [\Phi'(z)]^2 \\
				&= \frac{1}{A(z)^2 C(z)} \left( \frac{D_2(z)- 2 D_1(z)}{p(z)}A(z) + B(z)^2C(z) \right)
		\end{aligned}
\end{equation}
		As we mentioned, the poles of $D_1$ and $D_2$ are simple and contained in the set $\{x_1,\hdots, x_M\}$, which are all zeroes of $p^{-1}$. Thus, the expression inside the parenthesis is a polynomial. Using \eqref{eq:behD1} and \eqref{eq:behD2}, and the fact that $\deg p^{-1}=M=\deg B+\deg C$, we in fact obtain that the term between brackets is a polynomial of degree exactly $2\deg B+\deg C$, so $R$ is rational, and the same analysis also shows that (v) holds. Furthermore, poles of $\Phi'$ are necessarily poles of $R$ with doubled order, and simple poles may only come from the zeros of $A$ and $C$, that is, from the set $\Zcal\cup \Ccal$, and parts (i)--(iv) follow as claimed. 
	\end{proof}

With \eqref{QuadEq} at hand, we obtain qualitative info about the support and density of critical measures in a standard way.
	
	\begin{proposition}\label{QuadDiffChar}
Let $\mu\in \Mcal_1^\phi(\C)$ be a critical measure, with associated algebraic equation \eqref{QuadEq}. 	Then $\mu$ is supported on a union of analytic arcs, that are maximal trajectories of the quadratic differential $\varpi= -R(z) \dd z^2$ on $\overline\C$. Moreover, $\mu$ has no mass points, and in the interior of any arc of $\supp \mu$ it is absolutely continuous with respect to the arclength measure, with density
\begin{equation}\label{eq:densitycritmeasure}
			\dd \mu(s) = \frac{1}{\pi \ii} \sqrt{R(s)} \, \dd s,
\end{equation}
		where $\dd s$ is the complex line element.
		\end{proposition}
	\begin{proof}
	The key point here is that a critical measure $\mu$ satisfies an algebraic equation of the form \eqref{QuadEq}. Furthermore, by definition of the class $\Mcal_1^\phi(\C)$, the critical measure $\mu$ has finite log energy and therefore it cannot have mass points. Once these properties are known, the proof follows in the same way as in the proof of the analogous result \cite[Proposition~3.8]{Kuijlaars2015}, so we skip the details.
	\end{proof}

With Proposition~\ref{QuadDiffChar} at hand, we are able to establish some additional information on the support of critical measures.
	
	\begin{corollary}\label{cor:muIscompact}
		Let $\mu$ be a critical measure. The support of $\mu$ is a compact set made of a union of analytic arcs, that do not contain any of the zeroes of $Q$.
	\end{corollary}
	\begin{proof}
	Thanks to Proposition~\ref{QuadDiffChar}, we already know that $\supp\mu$ is a union of analytic arcs, and it remains to prove
\begin{enumerate}[(1)]
\item $\supp\mu$ is compact,
\item $\supp\mu$ does not contain poles of $\Phi$.
\end{enumerate}	
	
		Propositions \ref{prop:QuadraticEquation} and \ref{QuadDiffChar} tell us that $\mu$ is supported on maximal trajectories of the quadratic differential $-R(z)\dd z^2$. Still from Proposition~\ref{prop:QuadraticEquation}, $R(z)\dd z^2$ has a pole of order $2(N+1)\geq 4$ at $z_0\defeq\infty$ (with $N$ as in \eqref{eq:PhiDef}), a pole of order $2(m_j+1)$ at each finite pole $z_j$ of order $m_j$ of $\Phi$, double poles at each  $w_j$ of $\Phi$, possibly simple poles at elements of $\Ccal$, and no other poles. 

From the general theory of quadratic differentials \cite{Pommerenke1975}, at each pole $z_j$ there is neighborhood $U$ of $z_j$ such that every trajectory intersecting $U$ ends up at $z = z_j$ in at least one direction. In particular, if $\gamma \cap U \neq \emptyset$ for some maximal trajectory $\gamma$ of $\supp \mu$, then this arc $\gamma$ ends at $z_j$. On the other hand, from the explicit expression \eqref{eq:densitycritmeasure} for the density of $\mu$ 
and the fact that $R(s)$ has growth of order at least $4$ when $s \to z_j$ over the trajectories, we see that if $\gamma$ were to extend to a pole $z_j$, then $\mu$ would not be a finite measure. Analogously, near $w_i \in \Wcal$ whose residue of $\sqrt{-R(z)}$ is not purely imaginary, the density \eqref{eq:densitycritmeasure} would blow up linearly, and therefore $\mu$ would not be finite. This shows the claimed properties (1) and (2).


\end{proof}	

The quadratic differential $-R(z)\dd z^2$ has double poles at the points $w_j$'s, and around these points, as said, the trajectories consist of closed curves winding around $w_j$. As such, it could be possible that the support of a critical measure contains infinitely many of these closed trajectories. However, if in addition we assume that $\mu$ is also an equilibrium measure, then we can rule out such possibility, as claimed by the next proposition. For our final goal of understanding max-min problems, the critical measures of relevance will always be equilibrium measures as well, and the next proposition will become particularly useful.

\begin{proposition}\label{prop:critequilibrium}
	Assume that $\mu$ is a critical measure that is also an equilibrium measure on the external field $\phi = \Re \Phi$ for a set $S$. Suppose that the support of $\mu$ contains a connected component $\Gamma$ that is a trajectory that winds around a singularity $w_j$. Let $U$ be the bounded connected component of $\C \setminus \Gamma$. Then $\supp \mu \cap U = \emptyset$. In particular, a critical measure that is an equilibrium measure has at most one connected component winding around a  $w_j \in \Wcal$.
\end{proposition}

\begin{proof}

	Trajectories that wind around a singularity can only be found when $w_j$ is a double pole of $R$ with residue of $\sqrt{-R(z)}$ purely imaginary, and in this case the whole neighborhood $U$ is composed of trajectories with the same behavior. For such a neighborhood $U$, one has
    \[
        R(s) = \frac{R_0(s)}{(s - w_j)^2}, \quad R_0(w_j) \neq 0
    \]
    where $R_0$ is analytic on $U$, thus the function $\sqrt{R}$ is well defined and meromorphic on $U$.

    For $\delta > 0$ let 
    \[
        U_\delta = U \setminus B(z_0,\delta).
    \] 
    
    We claim that for sufficiently small $\delta < 0$, the set $\supp \mu \cap U_\delta$
    is composed of a finite amount of trajectories. To prove that first note that if $\eta$ is a trajectory contained in $U_\delta \cap \supp \mu$, then
    \[
        \mu(\eta) 
            = \int_{\eta} | \sqrt{R(s)} | \, | \dd s | 
            \geq 2 \pi \delta \cdot \inf_{s \in \overline{U_\delta}} |\sqrt{R(s)} |   > 0.
    \] 
      Therefore, if $\eta_n$ denotes an enumerable pair-wise disjoint family of trajectories contained in $U_\delta \cap \supp \mu$, then
    \[
        \mu(\C) 
            \geq \sum_{n \geq 1} \mu(\eta_n) 
            = +\infty
    \]
    a contradiction. This proves the claim.

    We now prove the proposition. Assume that $\supp \mu \cap U \neq \emptyset$. Since $\supp \mu$ is composed of trajectories, there exist $\eta \subset U$ a trajectory that is entirely contained in $\supp \mu$. Let $\Gamma'$ be the trajectory of $\supp \mu \cap U$ that is the closest to $\Gamma$, i.e., that there is no trajectory of $\supp \mu \cap U$ in the ``ring-like'' domain whose boundary is given by $\Gamma \cap \Gamma'$.

    For each $y \in \Gamma'$, let $\alpha = \alpha_y$ be the orthogonal trajectory passing through $y$. We see $\alpha$ as a oriented contour that intersects $\Gamma$ once at a point $x$, passes through $y$ and then goes to $w_j$. 
    
    The function
    \[
        f(z) = \int_x^z \sqrt{R(s)} \, \dd s .
    \]
    is well defined in the ``ring-like'' domain whose boundary is given by  $\Gamma \cup \Gamma'$.
    The identity
    \[
        \Re f(z) = U^\mu(z) + \phi(z) - l_x, \quad l_x = U^\mu(x) + \phi(x).
    \]
    immediately follows from the definition of $f$ and the algebraic equation satisfied by $C^\mu$, $\phi$ and $R$. 
    Since $\Im f$ is constant along $\alpha$, $\Re f$ is monotone along $\alpha$ by the Cauchy-Riemann equations, thus, for all $z \in \alpha$
    \[
        0 = \Re f(x) < \pm \Re f(z) = \pm \left[ U^\mu(z) + \phi(z) - l_x \right]
    \]
    where the sign depends on the monotonicity of $\Re f$. In any case, identity above contradicts the Euler-Lagrange variational conditions \eqref{eq:EulerLagrange} for $\mu$ when one takes $z = y$. Thus $\supp \mu \cap U = \emptyset.$

	
	
	
	 
\end{proof}

\begin{corollary}\label{cor:critequi}
	Let $\mu$ be a critical equilibrium measure. The support of $\mu$ is a compact set made of a finite union of analytic arcs.
\end{corollary}

\begin{proof}
	In Corollary \ref{cor:muIscompact} it was shown that the support of $\mu$ is compact, made of a union of analytic curves, composed of trajectories of the quadratic differential $-R(z) \dd z^2$, not containing any pole of $\Phi$, nor double poles of $R(z)$ whose residue of $\sqrt{-R(z)}$ is not purely imaginary.

	Since $R$ is a rational function, the remaining possible trajectories are critical trajectories that connect zeroes and simple poles of $R$, and trajectories that wind around a $w_j$. By the  Proposition \ref{prop:critequilibrium}, the amount of curves of the later case is bounded by $|\Wcal|$. 
\end{proof}

\subsection{Critical Sets}\label{sec:CritSets}

Recall that $\Hcal$ denotes the class of test functions introduced in \eqref{deff:testfunctions}.	

Variations $z \mapsto z + th(z)$ induce variations of measures, which we studied in detail in Section~\ref{sec:criticalmeasures}. In a similar way, it also defines variations of sets, that is, the analogue of Definition~\ref{deff:variationenergymeasure}, as we study next.
	
\begin{definition}\label{def:criticalset}
For $h\in \Hcal$ and a set $F\subset \overline\C$, let $F^t$ be the image of $F$ under the mapping $z\mapsto z+th(z)$, $t\in \R$.
We define the \emph{derivative of the energy functional at $F$ in the direction of $h$} by
	$$
		\D_h \I^\phi(F) \defeq \lim_{t \to 0}  \frac{\I^\phi(F^t) - \I^\phi(F)}{t},
	$$
	and say that a set $F$ is {\it critical} if $\D_h \I^\phi(F)=0$ for every $h\in \Hcal$.
\end{definition}	
	
It turns out that the derivative of the energy functional at a given set coincides with the derivative at the equilibrium measure of $F$.
	\begin{theorem}\label{teo:CritSet}
		Let $F \subset \C$ be a closed set with $\I^\phi(F) \in \R$, and $\mu(\phi,F)$ its equilibrium measure. Fix $h\in\Hcal$. Then $\D_h \I^\phi(F)$ exists and
		\begin{equation}\label{eq:critset}
			\D_h \I^\phi(F) = \D_h \I^\phi(\mu(\phi,F)).
		\end{equation}
	\end{theorem}
	
	This result was first presented in \cite{Rakhmanov2012}. The proof follows the ideas of \cite{Kuijlaars2015} very closely.  There is a minor mistake on the proof of Proposition 4.2 in \cite{Kuijlaars2015}, where it is said that if $h \in C^2_c$, then
	$$
		(x,y) \mapsto \frac{h(x) - h(y)}{x-y}
	$$ 
	is compactly supported. This is not true: for $x\in \supp h$ with $h(x)\neq 0$ and $y\in \C\setminus \supp h$, this quotient reduces to $h(x)/(x-y)\neq 0$. Before we move to the proof of Theorem~\ref{teo:CritSet} we solve this issue with Lemma \ref{lem:WeakConvDoubleIntegral} below.

	 We shall denote by $\Mcal_0$ the set of signed measures defined on $\C$, whose positive and negative parts belong to $\Mcal_1(\C)$ and have finite logarithmic energy. A sequence of finite signed measures $(\sigma_n)$ \emph{converges vaguely} (resp. weakly) to a finite signed measure $\sigma$ when
	$$
		\lim_{n \to \infty} \int f \, \dd \sigma_n = \int f \, \dd \sigma
	$$
	for every compactly supported continuous function $f$ (resp. for every bounded continuous function $f$). 
	
In what follows, we also talk about weak and vague convergence of a family of measures $(\mu_t)=(\mu_t)_{t\in \R\setminus\{0\}}$ as $t\to 0$ to a measure $\mu_0$. By such convergence, we mean that that $\mu_{t_n}\to \mu_0$ in the weak/vague sense, for any sequence of indices $(t_n)$ with $t_n\to 0$ as $n\to \infty$, and write $\mu_t\stackrel{t\to 0}{\to}\mu_0$ in the weak/vague sense.
	
	The following result is well known for compactly supported signed measures, and it was extended to allow for unbounded support in \cite{Kuijlaars2015}.
	
\begin{lemma}[{\cite[Lemma~4.3]{Kuijlaars2015}}]\label{lem:MutualEnergy}
	The functional
	$$
		\Mcal_0 \ni \sigma \mapsto \I(\sigma) = \iint \log |x-y|^{-1} \, \dd\sigma(x) \, \dd \sigma(y)
	$$
	is well-defined on $\Mcal_0$ and strictly positive for $\sigma \neq 0$. If $(\sigma_n)$ is a sequence of signed measures in $\Mcal_0$ with
	$$
		\lim_{n \to \infty} \I(\sigma_n) = 0,
	$$
	then the sequence $(\sigma_n)$ converges vaguely to the null measure.
\end{lemma}

We now study the convergence of one of the terms in \eqref{criteq} for weakly convergent sequences. The analysis is technical, and we split its core argument in a separate lemma.

\begin{lemma}\label{lem:WeakConvDoubleIntegral}
	Let $(\mu_n)$ be a sequence of Borel probability measures that vaguely converges to a measure $\mu_0$. Then, for every $h \in C_c^2$,
	$$
		\lim_{n\to\infty} \iint \frac{h(x) - h(y)}{x-y} \, \dd \mu_n(x) \, \dd \mu_n(y) = \iint \frac{h(x) - h(y)}{x-y} \, \dd \mu_0(x) \, \dd \mu_0(y).
	$$
\end{lemma}

\begin{proof}
	The function
	$$
		(x,y) \mapsto \frac{h(x) - h(y)}{x-y}
	$$
	is continuous and bounded.  Fix $R_0$ sufficiently large such that $\supp h \subset D_{R_0 - 1}$. For $R > R_0 + 1$ write
	\begin{multline*}
		\iint \frac{h(x) - h(y)}{x-y} \, \dd \mu_{n}(x) \, \dd \mu_{n}(y)  = \\
\iint_{ \substack{|x| < R \\ |y| < R }} \frac{h(x) - h(y)}{x-y} \, \dd \mu_{n}(x) \, \dd \mu_{n}(y)+ 2 \iint_{ \substack{|x| < R \\ |y| \geq R }} \frac{h(x)}{x-y} \, \dd \mu_{n}(x) \, \dd \mu_{n}(y)  
		\end{multline*}
For $R>0, n\in \N$, introduce the auxiliary quantities
\begin{equation*}
		 	C_{R,n}(x) 
		 	\defeq \int_{|y| \geq R} \frac{1}{x-y} \, \dd \mu_{n}(y),\quad C_{R,0}(x)\defeq \int_{|y| \geq R} \frac{1}{x-y} \, \dd \mu_{0}(y) 
\end{equation*}
which are continuous functions of $x\in  D_{R-1}\supset D_{R_0} $. 
Having in mind that $h\equiv 0$ in $D_R\setminus D_{R_0}$, we write
\begin{multline}\label{breakingh}
\left| \iint \frac{h(x) - h(y)}{x-y} \, \dd \mu_{n}(x) \, \dd   \mu_{n}(y) - \iint \frac{h(x) - h(y)}{x-y} \, \dd \mu_{0}(x) \, \dd   \mu_{0}(y) \right| \\
\begin{aligned}
\leq &  
		 			\left|
		 			\iint_{ \substack{|x| < R \\ |y| < R }} \frac{h(x) - h(y)}{x-y} \, \dd \mu_{n}(x) \, \dd \mu_{n}(y) - \iint_{ \substack{|x| < R \\ |y| < R }} \frac{h(x) - h(y)}{x-y} \, \dd \mu_{0}(x) \, \dd \mu_{0}(y)  
		 			\right|\\
&+ 2 \left| \int_{|x| < R_0} h(x) C_{R,n}(x) \, \dd \mu_{n}(x) - \int_{|x| < R_0} h(x) C_{R,0}(x) \, \dd \mu_{n}(x) \right| \\
&+ 2 \left| \int_{|x| < R_0} h(x) C_{R,0}(x) \, \dd \mu_{n}(x) - \int_{|x| < R_0} h(x) C_{R,0}(x) \, \dd \mu_0(x) \right|,
\end{aligned}
\end{multline}
and to finish the lemma we show that the right-hand side can be made arbitrarily small. 
		 	
We first examine the second term on the right-hand side, starting with the obvious inequality
		 	$$
		 		\left|\int_{|x| < R_0} h(x) C_{R,n}(x) \, \dd \mu_{n}(x) - \int_{|x| < R_0} h(x) C_{R,0}(x) \, \dd \mu_{n}(x) \right| \leq 
		 	\|h\|_\infty \sup_{|x| < R_0} |C_{R,n}(x) - C_{R,0}(x)| .
		 	$$
		 	Note that 
		 	$$
		 		|C_{R,n}(x) - C_{R,0}(x)| 
		 		\leq \int_{|w| \geq R} \frac{1}{|x-w|} \, \dd | \mu_{n} - \mu_0 |(w) 
		 	$$
		 	So, if $R > 2R_0$ and $|x| \leq R_0$, then $|x-w| \geq R - R_0 \geq \frac{R}{2}$ and
		 	$$
		 		|C_{R,n}(x) - C_{R,0}(x)| \leq \frac{4}{R}.
		 	$$
			This way, we conclude that for every $R>2R_0$, the inequality
			\begin{equation}\label{eq:vagueconv1}
				\left| \int_{|x| < R_0} h(x) C_{R,n}(x) \, \dd \mu_{n}(x) - \int_{|x| < R_0} h(x) C_{R,0}(x) \, \dd \mu_{n}(x) \right| 
				\leq \frac{4 \|h\|_\infty}{R},\quad \text{for every }n,
			\end{equation}
			is valid.

Next, for the third term in the right-hand side of \eqref{breakingh}, for any fixed $R>R_0+1$ we have
		 	\begin{equation}\label{eq:vagueconvhn2}
		 		\int_{|x| < R_0} h(x) C_{R,0}(x) \, \dd \mu_{n}(x) 
		 		\stackrel{n\to \infty}{\to} \int_{|x| < R_0} h(x) C_{R,0}(x) \, \dd \mu_0(x) \\
		 	\end{equation}
		 	because $h(x) C_{R,0}(x)$ is continuous on $D_{R-1} \supset D_{R_0}$ and with support on $D_{R_0}\supset \supp h$. 
		 	
Finally, to handle the first term in the right-hand side of \eqref{breakingh} we use some auxiliary measure-theoretic results. Thanks to Proposition~\ref{prop:appmeaszero} we can always choose $R>R_0+1$ for which $\mu(\partial D_R)=0$. In particular, using that $\partial (D_R\times D_R)\subset (\partial D_R\times \C)\cup (\C\times \partial D_R)$, we see that $\mu\times \mu (\partial (D_R\times D_R))=0$. From the vague convergence $\mu_n\times \mu_n\to \mu\times \mu$ and Proposition~\ref{prop:vagueconvergencerestriction}, we conclude
\begin{equation}\label{eq:vagueconvhn1}
\iint_{\substack{ |x|\leq R \\ |y|\leq R }} \frac{h(x)-h(y)}{x-y}\dd\mu_n(x)\dd\mu_n(y) \stackrel{n\to\infty}{\to} 
\iint_{\substack{ |x|\leq R \\ |y|\leq R }} \frac{h(x)-h(y)}{x-y}\dd\mu(x)\dd\mu(y).
\end{equation}

We now combine all the info we have given in a detailed way. Given $\epsilon>0$, we choose $R>2(R_0+1)$ such that
$$
\frac{8\|h\|_\infty}{R}<\frac{\epsilon}{3}.
$$
This way, $R=R(R_0,\epsilon,h)$, but nevertheless as mentioned we can make sure that $\mu(\partial D_R)=0$. When plugged into \eqref{eq:vagueconv1} this inequality ensures that 
$$
\left| \int_{|x| < R_0} h(x) C_{R,n}(x) \, \dd \mu_{n}(x) - \int_{|x| < R_0} h(x) C_{R,0}(x) \, \dd \mu_{n}(x) \right| 
				<\frac{\epsilon}{6},\quad \text{for every }n.
$$
For this fixed $R$, we now choose $n_1=n_1(R)=n_1(\epsilon)$ for which
$$
\left|\iint_{\substack{ |x|\leq R \\ |y|\leq R }} \frac{h(x)-h(y)}{x-y}\dd\mu_n(x)\dd\mu_n(y)-
\iint_{\substack{ |x|\leq R \\ |y|\leq R }} \frac{h(x)-h(y)}{x-y}\dd\mu(x)\dd\mu(y)\right| <\frac{\epsilon}{3},
$$
for every $n>n_1$, which is possible thanks to \eqref{eq:vagueconvhn1}. Next, and finally, we now choose $n_2=n_2(\epsilon)>n_1$ in such a way that
		 	\begin{equation}\label{breakingh3}
		 		\left| \int_{|x| < R_0} h(x) C_{R,0}(x) \, \dd \mu_{n}(x) - \int_{|x| < R_0} h(x) C_{R,0}(x) \, \dd \mu_0(x) \right| < \frac{\epsilon}{6},\quad \text{for} \quad n\geq n_2(R,\epsilon),
		 	\end{equation}
which is now possible due to \eqref{eq:vagueconvhn2}. Combining these last three inequalities, we see that for every $n\geq n_2$, the left-hand side of \eqref{breakingh} is bounded by $\epsilon$, which concludes the proof.

\end{proof}

	Finally, we have necessary results to prove Theorem \ref{teo:CritSet}.

\begin{proof}[{Proof of Theorem \ref{teo:CritSet}}]
For simplicity, for the rest of this proof we denote by $\mu_F=\mu(F,\phi)$ the equilibrium measure of a set $F$ in the (fixed) external field $\phi$, and by $\mu^{-t}$ the pushforward of the measure $\mu$ by the inverse mapping $h_t^{-1}$, $h_t(x)=x+th(x)$.

We start by writing
	$$
		\I^\phi(F^t) - \I^\phi(F) 
			= \I^\phi(\mu_{F^t}) - \I^\phi(\mu_F) 
			= \I^\phi(\mu_{F^t}) - \I^\phi(\mu_F^t) + \I^\phi(\mu_F^t) - \I^\phi(\mu_F). 
	$$
Recalling Definition~\ref{deff:variationenergymeasure},
	$$
		\lim_{t \to 0} \frac{\I^\phi(\mu_{F}^t) - \I^\phi(\mu_F)}{t} \eqdef \D_h \I^\phi(\mu_F).
	$$
Thus, \eqref{eq:critset} follows once we verify that 
	$$
		\I^\phi(\mu_{F^t}) - \I^\phi(\mu_F^t) = o(t),\quad t \to 0.
	$$
For $h \in C_c^2$, the expansions
	\begin{equation*}
		(h^{-1})_t(x) = x - th(x) + o(t), \quad \text{and}\quad 
		\frac{(h^{-1})_t(x) - (h^{-1})_t(y)}{x-y} = 1 - t \frac{h(x) - h(y)}{x-y} + o(t),
	\end{equation*}
are of simple verification, where both $o(t)$ terms are uniform in $x, y \in \C$. In practice, these approximations mean that we may treat the pushforward by $x\mapsto x+th(x)$ as being the same as the pushforward by $x\mapsto x-th(x)$, except for $o(t)$ terms. In other words, $(h^{-1})_t=h_{-t}+o(t)$. With this approximation in mind, we mimic the  proof of Lemma \ref{MuDerivative} (see for instance the key estimates \eqref{logbh} and \eqref{phibh}) and obtain
	$$
		\I^\phi(\mu_{F^t}^{-t}) - \I^\phi(\mu_{F^t}) = -t \D_h\I^\phi(\mu_{F^t}) + o(t), \quad t \to 0.
	$$
	On the other hand, from the very definition of $\D_h^\phi$ (see Definition~\ref{deff:variationenergymeasure}),
	$$
		\I^\phi(\mu_F^t) - \I^\phi(\mu_F) = t \D_h \I^\phi(\mu_F) + o(t), \quad t \to 0.
	$$
	The fact that $\mu_F^t$ is supported in $F^t$ and that $\mu_{F^t}^{-1}$ is supported on $F$ yields the inequalities $\I^\phi(\mu_F^t) - \I^\phi(\mu_{F^t}) \geq 0$ and $\I^\phi(\mu_F)  - \I^\phi(\mu_{F^t}^{-t})\leq 0$. Combined with the equations above, we obtain that as $t \to 0$,
	\begin{equation}\label{EnergyDifference}
		\begin{split}
			0 
				&\leq \I^\phi(\mu_F^t) - \I^\phi(\mu_{F^t}) \\
				&= \I^\phi(\mu_F)  - \I^\phi(\mu_{F^t}^{-t}) + t (\D_h \I^\phi(\mu_F) - \D_h \I^\phi(\mu_{F^t}) ) + o(t) \\
				&\leq t \left( \D_h \I^\phi(\mu_F) - \D_h \I^\phi(\mu_{F^t}) \right) + o(t),
		\end{split}
	\end{equation}
	Thus, to finish the proof, it is sufficient to show that $\D_h \I^\phi(\mu_F) - \D_h \I^\phi(\mu_{F^t}) = o(1)$ as $t \to 0$. 
	
Assume for a moment that $\mu_{F^t} \to \mu_F$ vaguely as $t \to 0$, fact which will be proved later. The function $h \Phi'$ belongs to $C_c^2$, so
		$$
			\lim_{t \to 0} \int h(x) \Phi'(x) \, \dd \mu_{F^t}(x) = \int h(x) \Phi'(x) \, \dd \mu_{F}(x).
		$$
		On the other hand, Lemma \ref{lem:WeakConvDoubleIntegral} guarantees that
	$$
		\lim_{t \to 0} \iint \frac{h(x) - h(y)}{x-y} \, \dd \mu_{F^t}(x) \, \dd \mu_{F^t}(y) 
		= \iint \frac{h(x) - h(y)}{x-y} \, \dd \mu_{F}(x) \, \dd \mu_{F}(y)
	$$
		
		By Proposition \ref{MuDerivative}, above limits imply
		$$
			\lim_{t \to 0} \D_h \I^\phi(\mu_{F^t}) = \D_h \I^\phi(\mu_F),
		$$
		which then concludes the proof.
		
Finally, it remains to verify that $\mu_{F^t}$ vaguely converges to $\mu_F$ as $t\to 0$. For that we will verify that $(\mu_F^t)$ converges to $\mu_F$ weakly, and that $\mu_{F^t} - \mu_F^t$ converges to the null measure vaguely.

		The first convergence is immediate from the definition of push-forward measure and the Dominated Convergence Theorem:
		$$
			\int f(x) \, \dd \mu_F^t(x) = \int (f \circ h_t)(x) \, \dd \mu_F(x) \stackrel{t \to 0}{\to} \int f(x) \, \dd \mu_F(x),
		$$
		for each bounded continuous function $f$.
		
		By Proposition \ref{MuDerivative}, the difference $D_h \I^\phi(\mu_F) - D_h \I^\phi(\mu_{F^t})$ is bounded. Thus, Equation \eqref{EnergyDifference} implies that $\I^\phi(\mu_F^t) - \I^\phi(\mu_{F^t}) \to 0$ as $t \to 0$. Now, by Lemma \ref{lem:MutualEnergy}, as $t \to 0$,
		\begin{equation*}
		\begin{split}
			0 \leq \I(\mu_F^t - \mu_{F^t}) 
			&= 2 \I^\phi(\mu_F^t) + 2 \I^\phi(\mu_{F^t}) - 4 \I^\phi \left(\frac{\mu_F^t + \mu_{F^t}}{2}\right) \\
			&\leq  2 \I^\phi(\mu_F^t) + 2 \I^\phi(\mu_{F^t}) - 4 \I^\phi(\mu_{F^t}) \\
			&= 2 \I^\phi(\mu_F^t) - 2 \I^\phi(\mu_{F^t}) \to 0,
		\end{split}
		\end{equation*}
		where the inequality above is obtaining by noting that $\frac{1}{2} ( \mu_F^t + \mu_{F^t}) \in \Mcal^\phi_1(F^t)$, so its weighted energy is larger than $\I^\phi(\mu_{F^t})$. Lemma \ref{lem:MutualEnergy} finally implies that $\mu_F^t - \mu_{F^t} \to 0$ vaguely as claimed.
\end{proof}

As a corollary of Theorem \ref{teo:CritSet},
	\begin{theorem}\label{teo:SummaryOfCriticalSection}
A set $F$ is critical in the sense of Definition~\ref{def:criticalset} if, and only if, its equilibrium measure $\mu(\phi,F)$ is critical in the sense of Definition~\ref{def:criticalmeasure}.
	\end{theorem}
In particular, if a set $F$ is critical, then its equilibrium measure $\mu(\phi,F)$ possesses all the properties that we obtained in Section~\ref{sec:criticalmeasures}. These properties will be fundamental later.

\section{Solving the max-min problem}\label{sec:MaxMinSolution}

	Having developed the theory of critical measures, we now go back to solving the max-min problem as posed in \eqref{eq:genPC}, for a class of contours $\Tcal$ for an admissible triple $(\Pcal(\Ccal),\Psi,\Pcal(\Theta_\infty))$. Following \cite{Kuijlaars2015}, such solution will be established in three steps as follows. 
	
The first step is to construct a class of sets $\Fcal_M$ which is different but related to our initial class of contours $\Tcal$. The advantage of $\Fcal_M$ will be that it is closed under the Hausdorff metric, but as a consequence it contains rather arbitrary closed sets, not only contours. We then solve a max-min problem on $\Fcal_M$, which will be essentially a consequence of Theorem~\ref{RakhmanovTheorem}, obtaining a max-min set $F_0$.
	
The second part, which is also the longest, consists of reconstructing a contour $\Gamma_0\in \Tcal$ from the set $F_0$, in such a way that their equilibrium measures coincide. This construction is done with a careful dropping of unnecessary parts of it, while preserving some geometry of the set, in a procedure that relies heavily on the results on critical measures that we previously obtained.

Finally, as the third and last step, we verify that $\Gamma_0$ is indeed a solution to the max-min problem as claimed by Theorem~\ref{maintheorem}.

\subsection{Finding a max-min candidate}

	A solution $\Gamma_0$ of the max-min problem should not intersect parts of the complex plane where the external field is very negative, since then $\Gamma_0$ would have very negative energy, and thus it cannot be the maximum over $\Tcal$. We exploit this in the following, presenting an equivalent max-min problem over a subfamily of $\Tcal$ whose closure on the Hausdorff metric satisfies the hypotheses of Rakhmanov's Theorem. 
	
	For $M$ large enough the set $\phi^{-1}(-M)$ is made of three types of curves (see Figure \ref{fig:TypesOfCurves}):
			\begin{enumerate}[(i)]
				\item $N$ disjoint analytic arcs, each of them stretching to infinity in its both ends;
				\item at each  pole $z_j$ of $\Phi$ of order $m_j$ there are $m_j$ ``tear-like'' curves, which are curves starting and ending at $z_j$;
				\item If $ \rho_k > 0$, there is a circle-like curve winding around $w_k$. 
			\end{enumerate}

	To verify above one writes $z-w = R e^{i \theta}$ and study what happens when $z \to w$, where $w$ is a singularity of $\Phi$. For example, when $w = w_k$, then 
	\[
		\phi(z) = \rho_k \log R  + g(z)
	\]
	where $g$ is analytic at $w_k$. Equating identity above to $-M$ and making $R \to 0$ we see that if $ \rho_k > 0$, the solution set is a curve winding around $w_k$, while if $ \rho_k < 0$, there is a neighborhood of $w_k$ in which the solution set is empty.

	\begin{figure}[!htb]
\minipage{0.3\textwidth}
  \includegraphics[width=\linewidth]{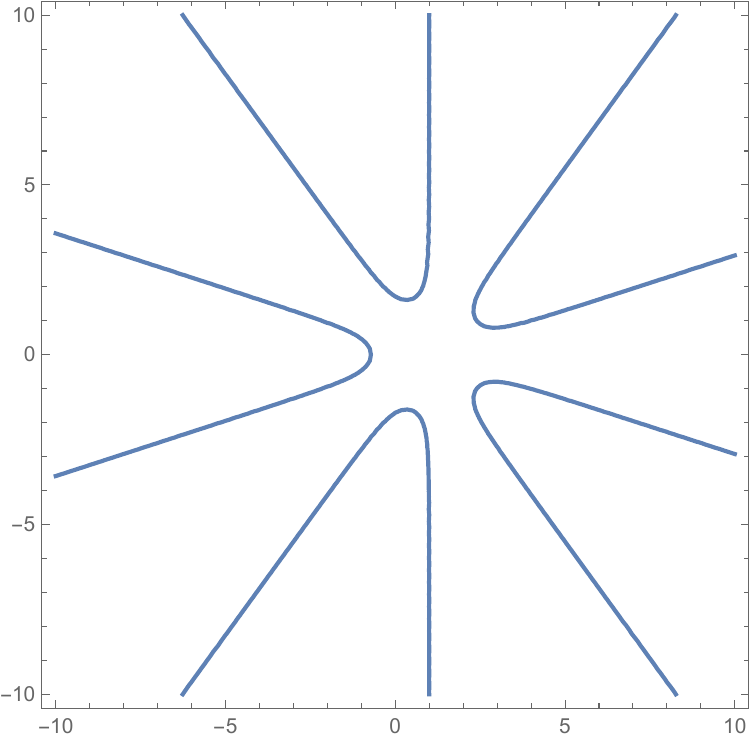}
  \caption*{Type (i) curves for $\Phi(z) = z^5-1$ and $M=15$.}\label{fig:Type1}
\endminipage\hfill
\minipage{0.3\textwidth}
  \includegraphics[width=\linewidth]{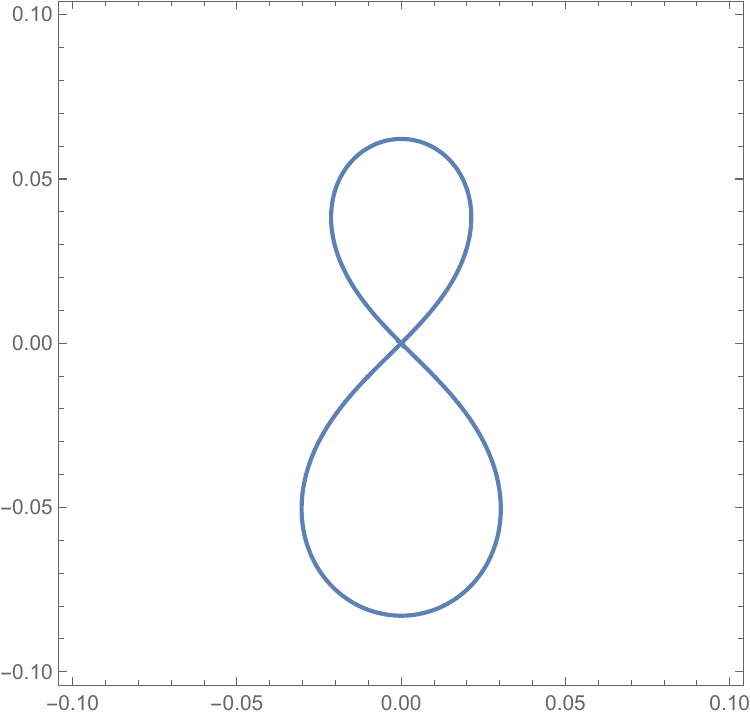}
  \caption*{Type (ii) curve for $\Phi(z) = \frac{(z-\ii)^4}{z^2}$ and $M = 200$.}\label{fig:Type2}
\endminipage\hfill
\minipage{0.3\textwidth}%
  \includegraphics[width=\linewidth]{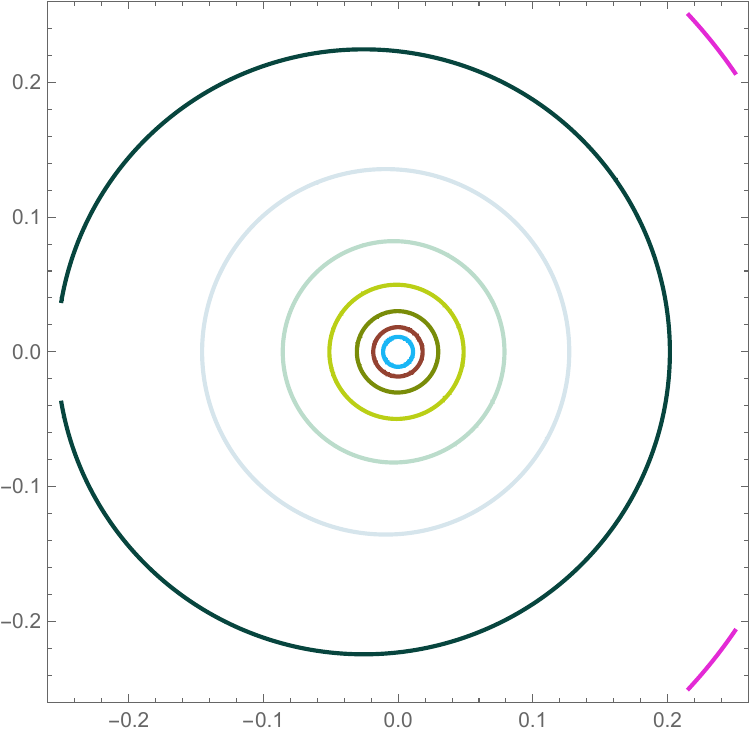}
  \caption*{Type (iii) curves for $\Phi(z) = z- 1 + 2 \log(z)$ and the choices $M=3,4,5,6,7,8$. }\label{fig:Type3}
\endminipage
\caption{Examples of type (i), (ii) and (iii) curves for different choices of $M$ and $\phi$. }\label{fig:TypesOfCurves}
\end{figure}

			Let $S=S_M$ be the union of the type (i) curves. Define
			\[
					r_M \defeq \inf \{ |z|; z \in S\}.
			\]
			For sufficiently large $M$, the set
			\[
				\Delta_{M} \defeq \phi^{-1}(-\infty,-M) \cap \{z \in \C; |z| > r_M \text{ and } \dist(z, S) > 8 \}
			\]
			is a non-empty open set whose boundary consists of a union of $N$ pairwise disjoint analytic arcs. 
			
			 Consider the subclass
			$$
				\Tcal_M \defeq \left\{ \Gamma \in \Tcal; \Gamma \cap \overline{\Delta_M} = \emptyset \right\}
			$$
			and its closure
			$$
				\Fcal_M \defeq \overline{\Tcal_M}
			$$
			with respect to the Hausdorff metric on closed subsets of $\C$ with the chordal distance. 
			
			The next step consists in verifying that the supremum of the weighted logarithmic energy over $\Tcal$ is the same as the supremum over $\Tcal_M$, this allows us to consider the max-min problem defined in $\eqref{eq:genPC}$ over $\Tcal_M$ instead of $\Tcal$.

			\begin{proposition}\label{prop:EquivalentMaxMin}
				For sufficiently large $M$, the properties
				\begin{equation}\label{eq:EquivalentMaxMin}
								\sup_{\Gamma\in \Tcal}\I^\phi(\Gamma)=\sup_{\Gamma\in \Tcal_M} \I^\phi(\Gamma) \leq \sup_{F \in \Fcal_M} \I^\phi(F)<+\infty
				\end{equation}
				hold true.
			\end{proposition}
			
			\begin{proof}
				The first inequality is always true, independent of $M$, simply because $\Tcal_M\subset \Fcal_M$. 
				

We start now verifying the second inequality. Recall that $\Ccal$ is a finite set and $\phi$ is finite at each point in $\Ccal$ (see \eqref{eq:polesfixedpts}), and that $\Zcal$ is the set of singularities of $\Phi$. For the rest of the proof, assume that $M>0$ satisfies  $-2M < \sup_{\Gamma \in \Tcal} \I^\phi(\Gamma)$ and  $-M < \inf \phi(\Ccal)$. At this stage we are allowing this supremum to possibly be equal to $+\infty$. 

Consider a simple closed curve $\gamma$ contained in $\C \setminus \Delta_M$ and intersecting all connected components of $\partial \Delta_M$ exactly once. Then $\C \setminus \gamma$ has two connected components. By making $M>0$ sufficiently large we make sure that the finite set $\Zcal\cup \Ccal$ is contained in the bounded connected component of $\C\setminus \gamma$. 

\begin{figure}
	\centering
	\includegraphics[scale=0.5]{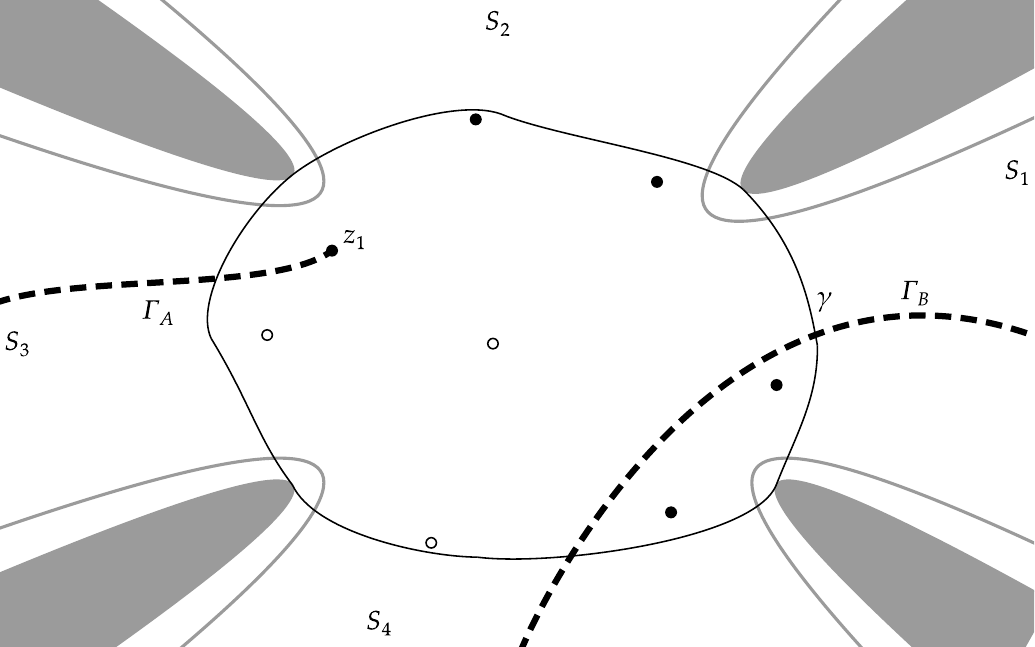}
	\caption{A choice for $\gamma$ when $N = 4$. Sets $\Delta_M$ and $\phi^{-1}(-M)$ are represented by the grey regions and the solid lines, and elements of the sets $\Ccal$ and $\Zcal$ are represented by black and white dots, respectively. If $\Tcal$ connects $z_1$ to infinity then each contour $\Gamma \in \Tcal$ contain a contour similar to $\Gamma_A$. Analogously, if $\Tcal$ connects infinity through $S_1$ and $S_4$, then $\Gamma$ contains a contour similar to $\Gamma_B$.  }
	\label{fig:DeltaM}
\end{figure}

Next, to continue, we split the analysis into two separate cases, namely when each $\Gamma\in \Tcal$ is unbounded, and the remaining case when we do not enforce contours $\Gamma\in \Tcal$ to be unbounded. 
In terms of Definition~\ref{def:AdmissibleClass}, we do not enforce contours $\Gamma \in \Tcal$ to be unbounded if, and only if, $\Pcal_\Psi = \Pcal_0(\Theta) = \emptyset$.
	
Thus, in the former case aforementioned, each $\Gamma \in \Tcal_M$ possesses an unbounded connected component $\wh\Gamma$ that stretches to infinity in a fixed admissible sector, meaning that $\wh\Gamma$ either connects a fixed point to infinity, or connects two distinct admissible sectors at infinity, see Definition~\ref{def:AdmissibleClass}--(v). The contour $\gamma$ is constructed in such a way that $\gamma\cup \partial \Delta_M$ separates $\infty$ from the points of the bounded connected component of $\C \setminus \gamma$, and also separates different admissible sectors from each other. Because contours in $\Tcal_M$ must not intersect $\overline\Delta_M$, it follows that $\wh\Gamma \cap \gamma \neq \emptyset$.  This property is preserved under taking closure with respect to the Hausdorff metric, that is, every $F \in \Fcal_M$ possesses a connected component $\wh F$ for which $\wh F \cap \gamma \neq \emptyset$.

				Now fix $R > 4$ such that $\gamma \subset D_{R-4}$. Then
				$$
					\wh F \cap \partial D_R \neq \emptyset
				$$
				by the same reasoning in which we obtained $\wh F \cap \gamma \neq \emptyset$. Consequently, there is a compact and connected set $F_0 \subset \wh F$ with
				$$
					F_0 \cap \partial D_R \neq \emptyset, \quad F_0 \cap \partial D_{R-4} \neq \emptyset, \quad F_0 \subset \{ R-4 \leq |z| \leq R \}.
				$$
				This implies that $\diam F_0 \geq 4$, thus  $\capp F_0 \geq 1$ which is equivalent to $\I(F_0) \leq 0$. Hence
				\begin{equation}\label{eq:estimate}
					\I^\phi(F) 
					\leq \I^\phi(F_0) 
					\leq \I(F_0) + 2\sup_{z \in F_0} \phi(z) 
					\leq 2\sup_{R - 4 \leq |z| \leq R} \phi(z)  
				\end{equation}
			The last term is finite and independent of $F \in \Fcal_M$, thus the inequality $\sup_{F \in \Fcal_M} \I^\phi(F) < +\infty$ follows.

			Now assume that contours in $\Tcal$ do not necessarily extend to $\infty$, this implies that there exists $\Pcal_0\subset \Pcal(\Ccal)$ with $|\Pcal_0|\geq 2$. We fix distinct points $c,\hat c\in \Pcal_0$, and choose $r > 0$ such that $\overline{D_{8r}(c)}\cap (\Zcal\cup \Ccal)=\{c\}$. 
			
Let $F \in \Fcal_M$. Start by assuming that there exists a uniformly bounded sequence $\Gamma_n \in \Tcal_M$ converging to $F$. Since $\Gamma_n$ has a connected component $\Gamma_0$ with $c,\hat c\in \Gamma_0$ and the sequence is uniformly bounded, then the same is true for $F$, say $F_0$. Since $|c-\hat c|\geq 8r$, we thus conclude that $\wh F \defeq  F_0 \cap \overline{D}_{8r}(c)$ is a compact connected subset of $F$ with $\diam \wh F \geq 8r$. As before, we now compute
			$$
				\I^\phi(F) 
				\leq \I^\phi(\wh F) 
				\leq - \log r + 2 \sup_{|z-c|\leq 8r} \phi(z).
			$$
			The last term is finite and independent of $F$. 
			
			If there is no uniformly bounded sequence $\Gamma_n$ with Hausdorff limit $F$, there exists a sequence $\Gamma_n$ converging to $F$ in the Hausdorff metric such that 
\[
	\Gamma_n \cap \partial D_R \neq \emptyset, \quad \Gamma_n \cap \partial D_{R-4} \neq \emptyset
\]
	and we can use the same arguments used to obtain the bound given by \eqref{eq:estimate}. This finishes the proof of the lemma.


We now prove the equality in \eqref{eq:EquivalentMaxMin}, it suffices to show that
\begin{equation}\label{eq:auxineqenergies}
\sup_{\Gamma\in \Tcal\setminus \Tcal_M}\I^\phi(\Gamma)<\sup_{\Gamma\in \Tcal} \I^\phi(\Gamma).
\end{equation}
Any $\Gamma \in \Tcal \setminus \Tcal_M$ intersects $\overline{\Delta_M}$ and thus, by condition (v) of Definition \ref{def:AdmissibleClass}, one of the following is true:

\begin{itemize}
	\item There is a connected contour $\gamma \subset \Gamma$ that intersects $\overline{\Delta_M}$ and connects a fixed point $c$ to infinity through one of the admissible sectors.
	\item There is a connected contour $\gamma \subset \Gamma$ that intersects $\overline{\Delta_M}$ and connects two fixed points $c_i$, $c_j$.
	\item There is a connected contour $\gamma \subset \Gamma$ that intersects $\overline{\Delta_M}$ and connects infinity through two admissible sectors.
\end{itemize}

In each of the cases there exists a point $x \in \gamma$ such that $\phi(x) > -M$. By following a contour starting at $x$ that intersects $\overline{\Delta_M}$, we obtain a subcontour $\gamma'$ such that
\begin{equation}\label{eq:gammaprime}
				\gamma' \cap S \neq \emptyset, \quad 
				\gamma' \subset \phi^{-1}(-\infty,-M) \quad 
				\text{and}\quad \gamma' \cap \overline{\Delta_M} \neq \emptyset.
				\end{equation}
				In particular, $\diam \gamma' > 8$, since $\dist(S, \overline{\Delta_M}) \geq 8$. 
The estimate of capacity in terms of diameter  $\capp(\gamma') \geq \frac{\diam \gamma'}{4} \geq 2$ (see \cite[Theorem 5.3.2]{Ransford}) implies that $\I(\gamma') \leq - \log 2$. Let $\omega$ be the Robin measure of $\gamma'$, that is, the minimizer of $\I$ over the class $\Mcal_1^{\phi\equiv 0}(\gamma')$. Then
\begin{equation}\label{eq:estenergygammaprime}
					\I^\phi(\Gamma) 
					\leq \I^\phi(\gamma') 
					\leq \I^\phi(\omega) 
					= \I(\omega) + 2\int_{\gamma'} \phi \, \dd \omega 
					\leq - \log 2 - 2M  
					<  \sup_{\Gamma' \in \Tcal} \I^\phi(\Gamma') - \log 2.
\end{equation}
				This holds for each $\Gamma \in \Tcal \setminus \Tcal_M$, which shows \eqref{eq:auxineqenergies} and concludes the proof of the claimed equality.

			\end{proof}				
					
\begin{remark}\label{rem:boundenergyint}
The argument in \eqref{eq:estenergygammaprime} shows that if a closed set $F$ is such that $F\cap \overline\Delta_M\neq \emptyset$, then it contains a subset $\gamma'$ satisfying \eqref{eq:gammaprime} and consequently
$$
\I^\phi(F)< \sup_{\Gamma\in \Tcal }\I^\phi(\Gamma).
$$
\end{remark}					
					
As a consequence, we conclude the first step in this section.
			\begin{corollary}\label{MaxCand}
				There exists $F_0 \in \Fcal_M$ such that 
				$$
					\I^\phi(F_0) = \sup_{F \in \Fcal_M} \I^\phi(F).
				$$
			\end{corollary}
			\begin{proof}

$\Fcal_M$ is compact in the Hausdorff metric, each one of its elements has at maximum $|\Pcal(C)| + |\Pcal_0(\Theta) \setminus \Pcal_\Psi|$ connected components and the supremum is finite by Proposition~\ref{prop:EquivalentMaxMin}, so the result follows from Theorem~\ref{RakhmanovTheorem}.
							
			\end{proof}

	Any $F_0 \in \Fcal_M$ who maximizes the energy functional on $\Fcal_M$ will be called a \emph{maximizing candidate}. Note that at the moment a maximizing candidate is not guaranteed to be a solution to our max-min problem of interest: convergence in the Hausdorff metric does not preserve contours, so $F_0$ is not necessarily an element of $\Tcal$.
	
Our next step is to ensure that maximizing candidates do not come to the ``boundary" of $\Fcal_M$, to ensure we are able to apply the variational theory developed in Section~\ref{sec:Critical}. Recall that the perturbation $F^t$ of a set $F$ was introduced in Definition~\ref{def:criticalset}.
			
\begin{lemma}\label{lem:MaximizationVar}
				Let $F_0 \in \Fcal_M$ be a maximizing candidate. For any given $h \in \Hcal$ there exists $t_0 > 0$ such that $F_0^t \in \Fcal_M$, for every $t \in (-t_0,t_0)$. In particular, every maximizing candidate is a critical set, whose equilibrium measure is compactly supported, and composed of finitely many analytic arcs.
\end{lemma}
			
		\begin{proof}
Thanks to Remark~\ref{rem:boundenergyint}, we know that $F_0 \cap \overline{\Delta}_M=\emptyset $ for any maximizing candidate $F_0\in \Fcal_M$. Since functions in $\Hcal$ are compactly supported, we learn that we can always choose $t_0>0$ sufficiently small such that $F_0^t \cap \overline{\Delta}_M = \emptyset$ for every $t \in (-t_0,t_0)$. On the other hand, by the very definition of $\Fcal_M$ as a closure, there is a sequence $(\Gamma_n) \subset \Tcal_M$ with $\Gamma_n \stackrel{n\to\infty}{\to} F_0$. It is straightforward to see that $\Gamma_n^t \to F_0^t$ as $n \to \infty$ as well, for each $t$ fixed. Since, as observed, $F_0^t\cap \overline{\Delta}_M = \emptyset$, and both of the sets on the left-hand side are closed in $\overline\C$, we learn that for each $t$ fixed there is $n_0=n_0(t)>0$ such that $\Gamma_n^t\cap \overline\Delta_M=\emptyset$ for every $n\geq n_0$. This shows that the sequence $(\Gamma_n^t)_{n\geq n_0}$ is in $\Tcal_M$, and therefore its limit $F_n^t$ belongs to $\Fcal_M=\overline\Tcal_M$.

			For every maximizing candidate $F_0$ and small enough $t$, $F_0^t$ is an element of $\Fcal_M$. Because the energy of $F_0$ is maximal within $\Fcal_M$, we obtain directly from Definition~\ref{def:criticalset} and Theorem~\ref{teo:CritSet} that $	D_h I^\phi(F_0) = 0$. Thus $F_0$ is a critical set, and the rest follows from Corollary \ref{cor:critequi}, concluding the proof.
\end{proof}		

\subsection{Constructing a solution for the max-min problem}

With Corollary~\ref{MaxCand}  we obtained the existence of $F_0 \in \Fcal_M$ that maximizes the energy functional. Remember that $F_0$ is not necessarily in $\Tcal$ as it may not be a finite union of contours. In this section we shall obtain a set $\Gamma_0 \in \Tcal$ such that $\mu_0(\phi,\Gamma_0) = \mu_0(\phi,F_0)$, and consequently $\I^\phi(F_0)=\I^\phi(\Gamma_0)$. Using Proposition~\ref{prop:EquivalentMaxMin}, we then obtain
$$
\I^\phi(\Gamma_0)\leq \sup_{\Gamma\in \Tcal}\I^\phi(\Gamma)\leq \sup_{F\in \Fcal_M}\I^\phi(F)=\I^\phi(F_0)=\I^\phi(\Gamma_0),
$$
so that the inequalities all turn into equalities, and $\Gamma_0$ is a solution to the max-min problem for $\Tcal$.

With $\mu_0=\mu_0(\phi,F_0)$, consider the set
$$
	\Lambda \defeq \left\{ z \in \C; U^{\mu_0}(z) +  \phi(z) > \ell_0 \right\}
$$
where $\ell_0$ is the Euler-Lagrange variational constant for $\mu_0$ (see \eqref{eq:EulerLagrange}). The new set $\Gamma_0$ will be constructed by connecting $\supp \mu_0$ to the fixed points and chosen admissible sectors through contours contained in $\overline{\Lambda}$.

Since $\mu_0$ is the equilibrium measure of $F_0$ in the external field $\phi$, the variational conditions \eqref{eq:EulerLagrange} imply that
		$$
			F_0 \subset \left\{ z \in \C; U^{\mu_0}(z) +  \phi(z) \geq \ell_0 \right\}.
		$$
		If $z \in F_0 \setminus \supp \mu_0$ then $U^{\mu_0} +\phi$ is harmonic near $z$, so it cannot have a local maximum in $z$. This implies that
		$$
			F_0 \subset \overline{\Lambda} \cup \supp \mu_0.
		$$

	The geometry of $\Lambda$ around infinity was described in \cite{Kuijlaars2015} in the context of a polynomial external field. The proof is local, in the sense that it only depends on the polynomial behavior of $\Phi$ near $z=\infty$. Thus, the same result is valid here, with the same proof, and we quote.

\begin{lemma}[{\cite[~Lemma 5.5]{Kuijlaars2015}}]\label{LambdaBehaviour}
  Given $\epsilon > 0$, define
\[
	S_j^\epsilon \defeq \left\{ z \in \C; |z| \geq R, |\arg z - \theta_j| \leq \frac{\pi}{2N} + \epsilon \right\}, \quad j=1,\ldots,N,
\]
where the angles $\theta_j$ were defined in \eqref{eq:deffsecSj}.
For each $\epsilon > 0$, there exists $R_\epsilon > 0$ such that for every $R > R_\epsilon$ the set $\Lambda \setminus D_R$ has exactly $N$ connected components $\lambda_1, \ldots, \lambda_N$, where
	$$
		L_{j,R} \defeq \left\{z = r\ee^{\ii \theta_j}; r \geq R \right\} \subset \lambda_j \subset S_j^\epsilon, \quad 1 \leq j \leq N.
	$$
	
\end{lemma}

	Now, we fix $\epsilon>0$, also ensuring that the sets $S_j^\epsilon$ are pairwise disjoint.
In addition, fix $R > R_\epsilon$ sufficiently large so that $\Zcal\cup\Ccal\cup \supp\mu_0\subset D_R$. 
	
	For $1 \leq j \leq N$, let $\Lambda(j)$ be the connected component of $\overline{\Lambda} \cup \supp \mu_0$ for which
	$$
		L_{j,R} \subset \Lambda(j).
	$$
	The set $\Lambda$ is open and has finitely many components whose boundary consists of a finite number of smooth arcs, whereas $\supp\mu_0$ consists of finitely many analytic arcs (see Lemma \ref{lem:MaximizationVar}). These properties ensure that $\Lambda(j)$ is pathwise connected, and we can connect any point of $\Lambda(j)$ to infinity through a curve entirely contained in $\Lambda(j)$ and stretching out to infinity in the sector $S_j$.
	
	The following lemma characterizes $\Lambda(j)$ in terms of the fixed connections. With a slight abuse of notation, for a fixed point $c\in \overline{\Lambda} \cup \supp \mu_0$ we denote by $\Lambda(c)$ the connected component of $\overline{\Lambda} \cup \supp \mu_0$ that contains $c$.
	
	\begin{lemma}\label{lem:Lambdaj}
	The following properties hold.
	\begin{enumerate}[(1)]
	\item If $\Tcal$ connects infinity through sectors $S_j$ and $S_k$ then $\Lambda(j) = \Lambda(k)$. 
	\item If $\Tcal$ connects a fixed point $c$ to infinity through $S_j$, then $\Lambda(c) = \Lambda(j)$. 
	\item If $\Tcal$ connects fixed points $c_1$ and $c_2$, then $\Lambda(c_1) = \Lambda(c_2)$.
	\end{enumerate}
	\end{lemma}

We postpone the proof of Lemma~\ref{lem:Lambdaj} to the following subsection, and now use it to construct $\Gamma_0$.

	\begin{proof}[Construction of $\Gamma_0$]
		Let $\epsilon$ and $R > R_\epsilon$ as previously described, and denote
		$$
			x_j = R \ee^{\ii \theta_j}, \quad 1 \leq j \leq N.
		$$
		From Lemma \ref{LambdaBehaviour} we see that $x_j \in \Lambda(j)$. 
		
		Let $C \in \Pcal(\Ccal)$ with $\Psi(C) = \emptyset$. Then Lemma \ref{lem:Lambdaj} implies that all points of $C$ are contained in the same connected component of $\overline{\Lambda} \cup \supp \mu_0$. So there is a connected compact set $\gamma_C \subset \overline{\Lambda} \cup \supp \mu_0$, which is a finite union of bounded piecewise smooth arcs, such that $C \subset \gamma_C$.

		For each $A \in \Pcal(\Ccal)$ with $\Psi(A) \neq \emptyset$, Lemma \ref{lem:Lambdaj} implies that all points of $A$ and $x_j$, with $\theta_j \in \Psi(A)$, are in the same connected component of $\overline{\Lambda} \cup \supp \mu_0$. There is a connected compact set $\gamma_A \subset \overline{\Lambda} \cup \supp \mu_0$ which is a finite union of bounded piecewise smooth arcs, such that $A \subset \gamma_A$ and $x_j \in \gamma_A$ for each $\theta_j \in \Psi(A)$. For $L_{j,R}$ as in Lemma \ref{LambdaBehaviour}, define
		$$
			\Gamma_A = \left[ \bigcup_{\theta_j \in \Psi(A)} L_{j,R} \right] \cup \gamma_A.
		$$
		Each $\Gamma_A$ is a finite union of piecewise smooth arcs, $A \subset \Gamma_A$, $\Gamma_A$ is connected and stretches out to infinity on each sector $S_j$ with $\theta_j \in \Psi(A)$. 
		
		Given $B \in \Pcal_0(\Theta) \setminus \Pcal_\Psi$, Lemma \ref{lem:Lambdaj} implies that the points $x_j$ with $j \in B$ belong to the same connected component of $\overline{\Lambda} \cup \supp \mu_0$. Then there exists a compact connected set $\gamma_B \subset \overline{\Lambda} \cup \supp \mu_0$, which is a finite union of bounded piecewise smooth arcs, satisfying $x_j \in \gamma_B$ for each $\theta_j \in B$.
		Define
		$$
			\Gamma_B = \left[ \bigcup_{\theta_j \in B} L_{j,R} \right] \cup \gamma_B.
		$$
		Each $\Gamma_B$ is a finite union of piecewise smooth arcs. $\Gamma_B$ is connected and stretches out to infinity on each sector $S_j$ with $\theta_j \in B$. 
		
		The union
		$$
			\Gamma_0 = 
				\left( \bigcup_{\substack{C \in \Pcal(\Ccal) \\ \Psi(C) = \emptyset}} \gamma_C  \right) 
				\bigcup \left( \bigcup_{\substack{A \in \Pcal(\Ccal) \\ \Psi(A) \neq \emptyset}} \Gamma_A
				 \right) 
				\bigcup \left( \bigcup_{B \in \Pcal_0(\Theta) \setminus \Pcal_\Psi} \Gamma_B \right)
		$$
		satisfies all requirements of Definition \ref{def:AdmissibleClass}, thus $\Gamma_0 \in \Tcal$. In particular,
		$$
			\I^\phi(\Gamma_0) \leq \sup_{\Gamma \in \Tcal} \I^\phi(\Gamma).
		$$
		Since $\Gamma_0 \subset \Gamma_0 \cup \supp \mu_0$, 
		\begin{equation}
			\I^\phi(\Gamma_0 \cup \supp \mu_0) \leq \I^\phi(\Gamma_0).
		\end{equation}
		On the other hand, since $\Gamma_0 \subset \overline{\Lambda} \cup \supp \mu_0$, the variational conditions (\ref{eq:EulerLagrange}) are valid for $\Gamma_0 \cup \supp \mu_0$, which means that $\mu_0$ is the equilibrium measure of $\Gamma_0 \cup \supp \mu_0$ in the external field $\phi$. So
		\begin{equation}
			\I^\phi(\Gamma_0 \cup \supp \mu_0) = \I^\phi(\mu_0) = \I^\phi(F_0)
		\end{equation}
		since $\mu_0$ is the equilibrium measure of $F_0$ as well. Recall that $F_0$ is a maximizer of the energy functional $\I^\phi$ over the class $\Fcal_M$ which contains $\Tcal_M$. Then, by Proposition \ref{prop:EquivalentMaxMin},
		$$
			\I^\phi(F_0) \geq \sup_{\Gamma \in \Tcal_M} \I^\phi(\Gamma) = \sup_{\Gamma \in \Tcal} \I^\phi(\Gamma).
		$$
		Combining inequalities above, we finally get
		$$
			\I^\phi(\Gamma_0) \leq \sup_{\Gamma \in \Tcal} \I^\phi(\Gamma) \leq \I^\phi(F_0) = \I^\phi(\Gamma_0 \cup \supp \mu_0 ) \leq \I^\phi(\Gamma_0)
		$$
		and so equality holds throughout these inequalities.
		
		Let $\mu_1$ be the equilibrium measure of $\Gamma_0$ in the external field $\phi$. Then
		$$
			\I^\phi(\mu_1) = \I^\phi(\Gamma_0) = \I^\phi(\Gamma_0 \cup \supp \mu_0).
		$$
		Then $\mu_1$ is a measure supported on $\Gamma_0 \subset \Gamma_0 \cup \supp \mu_0$ whose weighted energy is the minimum energy for measures on $\Gamma_0 \cup \supp \mu_0$. By uniqueness of the equilibrium measure it follows that $\supp \mu_0 \subset \Gamma_0$ and $\mu_0 = \mu_1$ is the equilibrium measure in the external field $\phi$ of $\Gamma_0$. 
		\subsection{Proof of Lemma \ref{lem:Lambdaj}}
				
		 Let $\Gamma_n$ be a sequence in $\Tcal_M$ that converges to $F_0$ in the Hausdorff metric.

		Let $R$ be large enough such that the circle $|z| = R$ intersects each of the components of $\Delta_M$ along $N$ circular arcs. Then there are $N$ circular arcs of $|z|=R$ outside of $\Delta_M$, that we denote by $\eta_1,\ldots,\eta_N$, where $\eta_k$ is in the direction of $S_k$. If $R$ is large enough, Lemma \ref{LambdaBehaviour} ensures that
		\begin{equation}\label{eq:etai}
			\left\{ z \in \eta_i; U^{\mu_0}(z) + \phi(z) \geq l_0 \right\} \subset \Lambda(i)
		\end{equation}
		for each $i$. We fix $R$ satisfying all conditions above and such that each fixed point $c$ is in $D_R$.
		
		We start by proving the second item. Let $c$ be a fixed point that $\Tcal$ connects to infinity through $S_j$. For each $n \in \N$, $\Gamma_n$ has a connected component that connects $c$ to infinity through $S_j$. On this connected component we can find a subset $\Gamma_n'$ which is a simple piece-wise $C^1$ contour connecting $c$ to $\infty$ through $S_j$. We consider $\Gamma_n'$ as an oriented contour, going from $c$ to $\infty$, and say that $\alpha$ lies before a point $\omega \in\Gamma_n'$ if $\alpha$ appears first when following the contour. By definition of $\Tcal_M$,
		$$
			\Gamma_n' \cap \Delta_M = \emptyset.
		$$				
		
		Now put $\eta = \cup_{i=1}^N \eta_i$. We follow the contour $\Gamma_n'$ starting at $c$. Put $\omega_{0,n} = c$. Let $\alpha_{1,n}$ be the first point of $\Gamma_n'$ after $\omega_{0,n}$ that intersects with $\eta$, then $\alpha_{1,n} \in \eta_{i_1}$ for some index $i_1$. If $i_1 = j$ we stop our process. Otherwise, since $\Gamma_n'$ goes to $\infty$ through $S_j$, there exists $\omega_{1,n}$ that is the last point of intersection of $\Gamma_n'$ with $\alpha_{i_1}$. After that, $\Gamma_n'$ enters $D_R$ and, since it must go to $\infty$ through $S_j$, it must intersect $\eta$ again. Let $\alpha_{2,n}$ be the first point after $\omega_{1,n}$ of this intersection, then $\alpha_{2,n} \in \alpha_{i_2}$ for some index $i_2 \neq i_1$. If $i_2 = j$ we are finished. Otherwise, there exists $\omega_{2,n}$ that is the last point of intersection of $\Gamma_n'$ with $\alpha_{i_2}$.

		The process is clear, and can be followed until there is some $l \in \{1,\ldots,N-1\}$ where $i_l = j$. This produces indexes $i_1,\ldots,i_l$, and points $c=\omega_{0,n}, \alpha_{1,n}, \omega_{1,n} \ldots, \omega_{l-1,n}, \alpha_{l,n}$ on $\Gamma_n'$ such that
		\begin{equation}
			i_l = j, \quad \alpha_{k,n}, \omega_{k,n} \in \eta_{i_k}, \quad 1 \leq k \leq l.	
		\end{equation}
		
		Denote by $\Gamma_n'(\omega, \alpha)$ the part of $\Gamma_n'$ strictly between $\omega$ and $\alpha$, then
		$$
			\Gamma_n'(\omega_{k,n}, \alpha_{k+1,n}) \subset D_R \setminus \Delta_M, \quad k=0, \ldots, l-1.
		$$	
		
		The numbers $i_1,\ldots,i_{l-1}$ depends on the contour $\Gamma_n$, so they also depends on $n$. Since these numbers are integers between $1$ and $N$, we can pass to a subsequence and assume that they are independent of $n$.
		
		The set
		$$
			\Gamma_n'[\omega_{k,n}, \alpha_{k+1,n}] = \Gamma_n'(\omega_{k,n}, \alpha_{k+1,n}) \cup \{ \omega_{k,n}, \alpha_{k+1,n} \} 
		$$
		is a connected closed subset of $\Gamma_n'$ contained in the compact set $\overline{D_R} \setminus \Delta_M$. By compactness of the Hausdorff metric, we can pass to a further subsequence and assume that there exists
		$$
			F_k' = \lim_{n \to \infty} \Gamma_n'[\omega_{k,n}, \alpha_{k+1,n}], \quad 0 \leq k \leq l-1.
		$$
		Note that $F_k'$ is a connected set since this property is preserved by taking limits in the Hausdorff metric if the sequence of sets is contained in a fixed compact set of $\C$. Of course, one also has
		$$
			F_i' \subset F_0
		$$
		since $\Gamma_n'[\omega_{k,n}, \alpha_{k+1,n}] \subset \Gamma_n$ for each $n$.
		
		By compactness of $\eta_i$ we can also assume (taking further subsequences) that $(\alpha_{i,n})_n$ and $(\omega_{i,n})_n$ converge to, say,
		$$
			\alpha_i = \lim_{n \to \infty} \alpha_{i,n}, \quad \omega_i = \lim_{n \to \infty} \omega_{i,n}.
		$$
		Then
		$$
			\omega_0 = c, \quad \alpha_k \in F_{k-1}' \cap \eta_k, \quad \omega_k \in F_k' \cap \eta_k.
		$$
		Then (\ref{eq:etai}) implies that
		$$
			\alpha_k, \omega_k \in \Lambda(i_k), 1 \leq k \leq l.
		$$
		Each $F_k'$ is connected subset of $F_0 \subset \overline{\Lambda} \cup \supp \mu_0$, so $\omega_k$ and $\alpha_{k+1}$ belong to the same connected component of $\overline{\Lambda} \cup \supp \mu_0$. Thus
		$$
			\Lambda(c) = \Lambda(i_1) = \Lambda(i_2) = \cdots = \Lambda(i_l) = \Lambda(j),
		$$
		which is exactly what we wanted to prove.
		
		
		The third case is similar. Let $c_1,c_2$ be fixed points that are connected by $\Tcal$. For each $n$ there exists a simple piece-wise $C^1$ contour $\gamma_n \subset \Gamma_n$ connecting $c_1$ to $c_2$. 
		
		Assume that there exists a subsequence of $\gamma_n$ that does not intersect $\eta$. We pass to this subsequence, and thus $\gamma_n$ is contained in $D_R$ for each $n \in \N$. Going to a further subsequence we can assume that $\gamma_n$ converges to a set $\gamma$ in the Hausdorff metric. This $\gamma$ is a connected set that contains $c_1$ and $c_2$, and is contained in $F_0 \subset \overline{\Lambda} \cup \supp \mu_0$. This obviously implies that $\Lambda(c_1) = \Lambda(c_2)$.
		
		If there is no subsequence as above then we can assume that $\gamma_n$ always intersect $\eta$. Then we do a similar construction as in the case of item $(2)$: Put $\omega_{0,n} = c_1$ and let $\alpha_{1,n}$ be the first point of intersection of $\gamma_n$ with $\eta$, then $\alpha_{1,n} \in \eta_{i_1}$ and, since $\gamma_n$ ends at $c_2$, there exists a last point of intersection $\omega_{1,n}$. If after $\omega_{1,n}$ the contour $\gamma_n$ does not intersect intersect $\eta$ anymore, we put $\alpha_{2,n} = c_2$ and we are done. Otherwise, let $\alpha_{2,n}$ be the first point of intersection of $\gamma_n$ with $\eta$, then $\alpha_{2,n} \in \eta_{i_2}$ and, since $\gamma_n$ ends at $c_2$, there exists a last point of intersection of $\gamma_n$ with $\eta_{i_2}$ that we denote by $\omega_{2,n}$. The process is clear, and we obtain a sequence $c_1 = \omega_{0,n}, \alpha_{1,n}, \omega_{1,n}, \ldots, \alpha_{i_l,n}, \omega_{i_l,n}, \alpha_{i_l+1,n} = c_2$. The rest of the proof follows by similar arguments as in the case of (2). Analogously we verify (1).
		
\end{proof}

\section*{Acknowledgments}

Both authors are grateful to Arno Kuijlaars and Andrei Martínez-Finkelshtein for valuable discussions. 


\section*{Declarations.}
	
	Figures were made with Wolfram Mathematica 13.3 and Mathcha.io.

	\noindent \textbf{Ethical Approval.}	Not applicable.
	
	\noindent \textbf{Competing interests.} The authors declare that they have no affiliations with or any type of involvement with any organization or entity with financial interest in the results discussed in this paper.
	
	\noindent \textbf{Author's contributions.} Victor Alves: Formal Analysis, Writing - original draft, Writing review and editing. Guilherme L.F. Silva: Formal Analysis, Writing - original draft, Writing review and editing.
	
	\noindent \textbf{Funding.} V. A. is supported by grant \# 2020/13183-0, São Paulo Research Foundation (FAPESP) and grant \# 2022/12756-1, São Paulo Research Foundation (FAPESP). 

G.S. acknowledges support by São Paulo Research Foundation (FAPESP) under Grants \# 2019/16062-1 and \# 2020/02506-2, and by Brazilian National Council for Scientific and Technological Development (CNPq) under Grant \# 315256/2020-6. 
	
	\noindent	\textbf{Availability of Data and materials.} Data sharing is not applicable to this article since no datasets were created nor analysed.

\appendix

\section{Some technical lemmas on measures}

In the course of the analysis of critical measures and critical sets, we needed some results on measure theory that are more or less a folklore. For completeness, we prove them here.

The first result we state is known as Portmanteau's Theorem. It is usually stated for weak convergence, but we need its analogue for vague convergence. It is valid for convergence of measures on more general metric spaces, but for us it suffices to state it in $\C$.

\begin{theorem}\label{thm:portmanteau}
Let $(\mu_n)$ be a sequence of Borel probability measures on $\C$, and $\mu$ a Borel measure on $\C$, not necessarily a probability measure. The following statements are equivalent.
\begin{enumerate}[(i)]
\item The sequence $(\mu_n)$ converges vaguely to $\mu$.
\item $\mu_n(A)\to \mu(A)$, for every Borel bounded set $A\subset \C$ for which $\mu(\partial A)=0$.
\end{enumerate}
\end{theorem}

A proof for this theorem can be found in \cite[Theorem 3.2, pg 52]{Resnick2007}. The first result we prove is a convergence statement for measures.

\begin{proposition}\label{prop:vagueconvergencerestriction}
Let $(\mu_n)$ be a sequence of Borel probability measures on $\C$, and $\mu$ a Borel measure on $\C$, not necessarily a probability measure. Suppose that $\mu_n\to \mu$ vaguely. If $B$ is a Borel set for which $\mu(\partial B)=0$, then the sequence $(\restr{\mu_n}{B})$ converges vaguely to $\restr{\mu}{B}$.
\end{proposition}

\begin{proof}
Let $A$ be a Borel set for which $\mu(\partial A\cap B)=\restr{\mu }{B}(\partial A)=0$. We will show that 
\begin{equation}\label{eq:convrestrmun}
\restr{\mu_n}{B}(A)\to \restr{\mu}{B}(A),
\end{equation}
and the result will then follow from \eqref{thm:portmanteau}.

We denote by $\inter E, \overline E$ and $E^c$ the interior, topological closure and complement of a set $E$, respectively. Write
$$
A=(A\cap \inter B)\cup (A\cap \partial B)\cup (A\cap \overline B^c).
$$
This is a union over disjoint sets, and $(A\cap \overline B^c)\cap B\subset \overline B^c\cap B=\emptyset$, so
$$
\restr{\mu_n}{B}(A)=\restr{\mu_n}{B}\left(A\cap \inter B\right)+\restr{\mu_n}{B}\left(A\cap \partial B\right),
$$
and likewise
$$
\restr{\mu}{B}(A)=\restr{\mu}{B}\left(A\cap \inter B\right)+\restr{\mu}{B}\left(A\cap \partial B\right).
$$
Using the vague convergence of $(\mu_n)$ towards $\mu$, we obtain
$$
0\leq \restr{\mu_n}{B}\left(A\cap \partial B\right) \leq \mu_n(\partial B)\to \mu(\partial B)=0.
$$
Also,
$$
0\leq \restr{\mu}{B}(A\cap \partial B)\leq \mu(\partial B)=0,
$$
and the convergence statement \eqref{eq:convrestrmun} becomes equivalent to proving that
\begin{equation}\label{eq:convrestrmun2}
\restr{\mu_n}{B}\left(A\cap \inter B\right) \to \restr{\mu}{B}\left(A\cap \inter B\right).
\end{equation}
Since $\partial (E_1\cup E_2)\subset \partial E_1\cup \partial E_2$ for arbitrary sets $E_1,E_2$, we also have that 
$$
0\leq \mu(\partial (A\cap \inter B))\leq \mu(\partial A\cap \partial B)\leq  \mu(\partial B)=0.
$$
Using Theorem~\ref{thm:portmanteau},
$$
\restr{\mu_n}{B}\left(A\cap \inter B\right)=\mu_n(A\cap \inter B)\to \mu(A\cap \inter B).
$$
Thus, \eqref{eq:convrestrmun2} is proven, and the proof is complete.
\end{proof}

The second result is a simple statement on measures.

\begin{proposition}\label{prop:appmeaszero}
Let $\mu$ be a finite Borel measure on $\C$. For each $R_0>0$ given, there exists $R>R_0$ for which
$$
\mu(\partial D_R)=0.
$$
\end{proposition}
\begin{proof}
Suppose the contrary. Then there exists $R_0>0$ such that
\begin{equation}\label{auxeqapp1}
\mu(\partial D_R)>0,\quad \text{for every } R>R_0.
\end{equation}
Consider now the family of sets
$$
A_m\defeq \left\{ R>R_0 \; ; \; \mu(\partial D_R)\geq 1/m \right\}, \quad m\in \N.
$$
From \eqref{auxeqapp1} we learn that $(R_0,+\infty)\subset \cup_{m\geq 1}A_m$. In particular, at least one of the sets $A_m$ is uncountable, say $A_{m_0}$, and therefore it has infinitely many members. We extract a sequence $(R_n)$ of pairwise distinct elements in $A_{m_0}$, and compute
$$
\mu\left(\bigcup_{n\geq 1}\partial D_{R_n}\right)=\sum_{n=1}^\infty \mu(\partial D_{R_n})\geq \sum_{n=1}^\infty \frac{1}{m}=+\infty,
$$
which is in contradiction with the fact that $\mu$ is a finite measure.
\end{proof}

	\section{Quadratic differentials}\label{sec:QDs}
	
On this section we briefly discuss relevant properties on quadratic differentials in $\overline \C$. For a full presentation on quadratic differentials, the interest reader is referred to the book by Strebel \cite{Strebel1984}, and also Jensen's chapter on Pommerenke's book \cite[Chapter 8]{Pommerenke1975}. 

	A rational function $R$ defines on the Riemann sphere $\overline{\C}$ the quadratic differential $-R(z) \dd z^2$. 
	
	The \emph{critical points} of $R$ are the zeroes and poles of $R$, including possibly the point $z=\infty$. A critical point $p \in \C$ has \emph{order}
	\begin{itemize}
		\item $n > 0$ if $p$ is a zero of $R$ of order $n$;
		\item $-n < 0$ if $p$ is a pole of $R$ of order $n$.
	\end{itemize}
	A non-critical point has order $0$. The point $\infty$ has order
	$$
		\order (\infty) = -4 - \sum_{p \in \C} \order (p).
	$$
	A critical point $p$ is \emph{finite} if $\order(p) \geq -1$.
	
	An arc $\gamma$ is called an \emph{horizontal arc} of $-R(z) \dd z^2$ if, along $\gamma$,
	$$
		-R(\alpha(t)) \dd (\alpha(t))^2 > 0
	$$
	where $\alpha$ is any parametrization of $\gamma$.
	At a regular point, this condition is locally stated as 
	$$
		\Real \int \sqrt{R(s)} \, \dd s \equiv \text{const}, \quad z \in \gamma.
	$$
	A maximal horizontal arc is called a \emph{horizontal trajectory}, or simply a \emph{trajectory}, of $-R(z) \dd z^2$. A trajectory that ends at two finite critical points is called a \emph{critical trajectory}.
	
	The local structure of trajectories is as follows:
	\begin{itemize}
		\item If $p$ is a zero of order $n \geq 1$, then there are exactly $n+2$ trajectories emanating from $p$, and they do so at equal consecutive angles $\frac{2\pi}{n+2}$;
		\item Any regular point $p$ belongs to an open arc of exactly one trajectory, not necessarily critical.
		\item From a simple pole $p$, there is exactly one trajectory emanating from $p$;
		\item From a double pole $p$, the local behavior depends on the \emph{residue} of $p$, which is defined as the residue $c$ of the function $\sqrt{-R(z)}$ at $z = p$, well defined up to a sign. If $c$ is purely imaginary, there is no trajectory passing through $p$, and  trajectories near $p$ consist of closed curves winding around $p$. Otherwise, every direction has a trajectory emanating from $p$.
	\item From a pole $p$ of order $2(k+1)$, $k > 0$, there are $2k$ directions, forming equal angles at $p$, such that any trajectory ending at $p$ does so in one of these directions. There is also a neighborhood $U$ of $p$ such that any trajectory entering $U$ ends at the pole in at least one direction, and trajectories entirely contained in $U$ end up at $p$ in two consecutive directions;
	\end{itemize}

In this paper, we used quadratic differentials of a particular rational structure, namely of the form
	$$
		R(z) = \left(\Phi'(z) \right)^2 + \frac{D(z)}{A(z)C(z)}, \quad \Phi'(z) = \frac{B(z)}{A(z)}, \quad \gcd(A,B) = 1,
	$$
	where $D$ is a polynomial of degree $\deg B + \deg C - 1$, $\Phi$ is given by \eqref{eq:PhiDef} and $C$ is as in \eqref{eq:deffC}. In terms of the structure of zeros and poles, the following properties hold:
	\begin{itemize}
		\item If $z_j$ is a pole of $\Phi$ of order $m_j$, then $z_j$ is a pole of $-R(z)\dd z^2$ of order $2(m_j+1)$;
		\item Each  $w_j$ that is different from each $z_i$ is a double pole of $-R(z)\dd z^2$, and the residue of $\sqrt{-R(z)}$ at $w_j$ is exactly $\rho_j$.
		\item Each fixed point $c_j$ is a simple pole of $-R(z)\dd z^2$ if $D(c_j)\neq 0$, a regular point if $c_j$ is a simple zero of $D$, or a zero of order $\beta$ of $-R(z)\dd z^2$, in case $c_j$ is a zero of order $\beta+1\geq 2$ of $D$;
		\item The point $p=\infty$ is a pole of order $2(N+1)$ of $-R(z)\dd z^2$, where $N = \deg P - \deg Q$.
	\end{itemize}

\bibliographystyle{abbrv}  

\end{document}